\documentclass[a4paper,12pt,oneside,reqno]{amsart}

\usepackage[latin1]{inputenc}
\usepackage{hyperref}
\usepackage[headinclude,DIV13]{typearea}
\areaset{16cm}{24.7cm}
\usepackage{txfonts,amssymb,amsmath,amsthm,bbm,wasysym}
\usepackage[pdftex]{graphicx}
\usepackage{geometry}
\usepackage[framemethod=tikz]{mdframed}
\usepackage{cancel} 
\usepackage{xcolor}
\usepackage{graphicx} %, psfrag
\graphicspath{ {figures/} }
\usepackage{caption}
\usepackage{subcaption}
\usepackage{mathrsfs}
\usepackage{xcolor}
\usepackage{enumerate}
\usepackage{nicefrac}	
\theoremstyle{plain}
\newtheorem{theorem}{Theorem}[section]
\newtheorem{lemma}[theorem]{Lemma}

\newtheorem{corollary}[theorem]{Corollary}

\newtheorem{example}{Example}

\theoremstyle{definition}
\newtheorem{definition}[theorem]{Definition}
\newtheorem{assumption}{Assumption}

{\itshape}{\rmfamily}
{\itshape}{\rmfamily}

\theoremstyle{remark}
\newtheorem{remark}{Remark}

\def\paragraph#1{\noindent \textbf{#1}}

\numberwithin{equation}{section}

\def\d{\mathrm{d}}
\def\a{\alpha}
\def\b{\beta}
\def\e{\e} %nonsense

\def\g{\gamma}
\def\t{\tau}
\def\G{\Gamma}
\def\R{{\Bbb R}}  %%
\def\N{{\Bbb N}}  %
\def\P{{\Bbb P}}  %% carateri piu belle per campi di nombre

\let\cal=\mathcal
\def\EE{{\cal E}}
\def\GG{{\cal G}}
\def\LL{{\cal L}}
\def\MM{{\cal M}}
\def\OO{{\cal O}}
\def\SS{{\cal S}}
\def\VV{{\cal V}}
\def\eps{\varepsilon}
\def\v{\mathbf{v}}
\def\w{\mathbf{w}}

\def\liv{\text{living}}
\def\mut{\text{mut}}
\def\fix{\text{fix}}
\def\lalpha{{\lfloor\alpha\rfloor}}
\DeclareMathOperator*{\argmin}{arg\,min}

\newcommand{\ifct}[1]{\mathbbm{1}_{ {#1} }}
\newcommand{\Exd}[1]{\mathbb{E}\left[ {#1}\right]}
\newcommand{\Prob}[1]{\mathbb{P}\left({#1}\right)}
\newcommand{\abs}[1]{\left\lvert#1\right\rvert}
\newcommand{\dset}[1]{\left\{ {#1} \right\}}
\newcommand{\ESC}{\text{\hspace{-0.15em}E\hspace{-0.05em}S\hspace{-0.1em}C}}
\begin{document}

\title[\ ]{A general multi-scale description of metastable adaptive motion across fitness valleys}

\author[]{Manuel Esser, Anna Kraut}
\address{M.\ Esser\\Institut f\"ur Angewandte Mathematik\\
	Rheinische Friedrich-Wilhelms-Universit\"at\\ Endenicher Allee 60\\ 53115 Bonn, Germany}
\email{manuel.esser@uni-bonn.de}
\address{A.\ Kraut\\School of Mathematics\\
	University of Minnesota - Twin Cities\\ 206 Church St SE\\ Minneapolis\\ MN 55455\\ USA}
\email{kraut082@umn.edu}

\keywords{adaptive dynamics, stochastic individual-based models, birth death processes with immigration, metastability, multi-scale limits}

\subjclass[2010]{37N25, 60J27, 60J80, 92D15, 92D25}

\thanks{This work was partially supported by the Deutsche Forschungsgemeinschaft (DFG, German Research Foundation) under Germany's Excellence Strategy GZ 2047/1, Projekt-ID 390685813 and GZ 2151, Project-ID 390873048 and through Project-ID 211504053 - SFB 1060. The authors thank Anton Bovier for stimulating discussion and feedback on the manuscript. Moreover, the authors are very grateful to the anonymous referees for their numerous comments and questions.}

\begin{abstract}
	{We consider a stochastic individual-based model of adaptive dynamics on a finite trait graph $G=(V,E)$. The  evolution is driven by a linear birth rate, a density dependent logistic death rate and the possibility of mutations along the (possibly directed) edges in $E$. We study the limit of small mutation rates for a simultaneously diverging population size. Closing the gap between \cite{BoCoSm19} and \cite{CoKrSm21} we give a precise description of transitions between evolutionary stable conditions (ESC), where multiple mutations are needed to cross a valley in the fitness landscape. The system shows a metastable behaviour on several divergent time scales, corresponding to the widths of these fitness valleys.
	We develop the framework of a meta graph that is constituted of ESCs and possible metastable transitions between those. This allows for a concise description of the multi-scale jump chain arising from concatenating several jumps.
	Finally, for each of the various time scale, we prove the convergence of the population process to a Markov jump process visiting only ESCs of sufficiently high stability.}
\end{abstract}

\maketitle
% \tableofcontents		

\section{Introduction}

The theory of evolution aims to understand the adaptation of biological populations to their environment through mutation and selection. Following the principles originally proposed by Darwin, it associates to each individual a fitness, which characterises their ability to survive and produce a growing population. The path of evolution, tracing the types of individuals that were able to fixate in the population, usually follows a sequence of types of increasing fitness.
However, in many cases the mutational path has to pass through a number of deleterious or neutral intermediate types in order to reach a type of higher fitness. This can for example be seen in cancer initiation, where multiple driver mutations need to be accumulated to induce an outgrowing population \cite{MarRai17}. 
Other examples are the formation of complex mechanisms like flagella in bacteria, where only partially functional intermediate stages of flagella yield an evolutionary disadvantage but fully functional apparatuses lead to increased fitness \cite{PaMa06}.
See also \cite{DeViKru14} for a review of empirical fitness landscapes arising in nature.

When the population needs to cross types of lower fitness in order to reach a fitter type, many such attempts will be unsuccessful. This is because the intermediate unfit types are destined to go extinct within a short time and might not produce a new mutant type before this happens. As a result, the waiting time to cross a valley in the fitness landscape is much longer than the invasion time of fit mutant types that are directly accessible. Once a fit type is attained, however, it rapidly fixates in the population. These dynamics, which can also be analysed in the framework of metastability, as illustrated below, have already been studied heuristically by Gillespie in the 80s \cite{Gill84}.  Since then, fitness valleys have been studied in a variety of mathematical models, ranging from Moran models \cite{Koma07,GoIwNoTr09} to multi-type branching processes \cite{NichAn19}. 

The model that we want to focus on in this paper is a stochastic individual-based model of adaptive dynamics, for which Bovier, Coquille and Smadi have studied fitness valleys in the simple case of a linear trait space \cite{BoCoSm19}.
This type of model tracks the sizes of different subpopulations and - opposed to many others like the Moran model - does not work under the assumption of a constant overall population size.  It is in this aspect closer to branching processes, where the population size varies over time. However, infinite growth is limited due to competitive interactions. Moreover, selective advantages of certain traits are not prescribed by a fixed parameter but arise through these interactions. This is particularly important for the long-term evolution of the population since the fitness landscape depends on the current composition of the dominant population and changes over time.

This study of the interplay of ecology and evolution goes back to ideas from Metz and Geritz (among others) in the early 90s \cite{MeNiGe92}. Shortly after, an individual-based approach has been proposed by Bolker and Pacala \cite{BoPa97} and a rigorous construction was first presented by Fournier and M\'el\'eard almost 20 years ago \cite{FoMe04}. Since then, these models have been the topic of study for scaling limits in a variety of parameter regimes and extensions to the base model (e.g.\ \cite{Cha06,ChMe11,BaBoCh17,Sma17,BoCoNeu18,KrBo19,ChMeTr19,CoKrSm21}). We refer to \cite{Bov21} for a comprehensive overview of various scaling limits.

To study the typical long-term behaviour of the population, two scaling parameters are introduced: The carrying capacity $K$, which scales the order of the population size, and the mutation probability $\mu_K$, which scales the frequency of mutation events. For large populations ($K\to\infty$) and rare mutations ($\mu_K\to0$), different mechanisms that change the state of the population - like mutations introducing a new type or interactions between individuals that lead to a new equilibrium state of resident traits - act on different time scales. There are three important time scales in this setting: Ecological interactions between well-established subpopulations, like the competition for resources, can change the composition of the overall population within a short time of order 1. This is related to classical Lotka-Volterra dynamics and leads to equilibrium states between the larger traits. Short-range mutations and the initial exponential growth of small mutant populations can be witnessed on a logarithmic time scale of order $\ln K$. Finally, long-range mutations - in particular those that need to traverse a large fitness valley of width $L$ - are quite rare and occur on a time scale of order $1/K\mu_K^L$. The distinction between long and short-range mutations depends on the choice of the mutation probability $\mu_K$, where long ranges $L$ satisfy $K\mu_K^L\ll1$. To obtain a non-trivial limit as $K\to\infty$, the population size is usually rescaled by $K$. As a result, only the established resident traits are visible. Since the ecological changes of these traits happen very fast in comparison with the other time scales, the limit of the population process yields a jump process that transitions between different equilibrium states.

The effects of short-range mutations on the $\ln K$-time scale have been studied extensively by Coquille, Kraut, and Smadi in \cite{CoKrSm21}. The authors give a full description of the limiting dynamics for the scenario of a general finite graph as a trait space. As mentioned above, the crossing of fitness valleys through long-range mutations (on the $1/K\mu_K^L$-time scale) has been analysed for a simple linear trait space in \cite{BoCoSm19}. Moreover, the case of very rare mutations, where even neighbouring traits are regarded as long-range mutations, has already been studied by Champagnat and M\'el\'eard in \cite{Cha06,ChMe11}, who showed convergence to the trait substitution sequence or polymorphic evolution sequence.

The present paper finally closes the gap between the previous works and gives a full description of the jump processes resulting from long-range mutations on general finite trait graphs, thus extending the results of \cite{BoCoSm19} to the more general setting of \cite{CoKrSm21}. This general setting entails that, for a given equilibrium state, there might be several paths to cross the surrounding fitness valley. Concentrating on the decisive, shortest paths we calculate the rate of a transition to the next evolutionary stable condition and give the precise asymptotics in Theorem \ref{Thm:main} and Corollary \ref{Cor:main}. The length of the shortest paths determines the time scale to cross the valley. Based on this, we introduce the notion of a stability degree $L$ to classify the equilibrium states. Combining multiple of these steps gives rise to a jump chain that moves on a so called metastability graph stated in Corollary \ref{Cor:kProcess}. This graph typically consists of fitness valleys of different width, which can be crossed on different time scales of the form $1/K\mu_K^L$. Depending on the choice of time scale, only some of these transitions are possible (valleys of width strictly larger than $L$ cannot be crossed) or visible (transitions of valleys of width strictly smaller than $L$ are immediate). This leads to different limiting jump processes  in Theorem \ref{Thm:Lscale}.

When long-range mutations are necessary to cross a large fitness valley, the system displays an almost stable behaviour on shorter time scales but can change its state when waiting a long time. This type of phenomenon is also known as \textit{metastability}. It has been studied mathematically mostly in the context of physics and statistical mechanics (e.g.\ \cite{CiNa13}). However, the concept is very versatile and can be applied to many dynamical systems, including models for biological processes.
This has for example been mentioned in \cite{BoCoSm19} for models of adaptive dynamics,
and in \cite{DawGre14} for population dynamics.

In the former case, as well as in this paper, the role of the traditional physical energy (landscape) is taken over by the fitness (landscape). Instead of passing a critical state of high energy, the process has to cross a valley of negative fitness through a sequence of deleterious mutations. Similarly to the fast dynamics after passing a high energy state, the adaptive dynamics system quickly attains a new metastable equilibrium once a fit mutant is reached due to fast exponential growth. The results of \cite{BoCoSm19} and this paper even confirm classical definitions of the mean time for a metastable transition (e.g.\ \cite{BovHol15}),  by proving that the waiting times for jumps between equilibrium states are exponentially distributed when considering the correct time scale.

While single jumps across a fitness valley can be regarded as metastable transitions, the limiting jump chain can be related to the concept of \textit{adaptive walks} or \textit{flights}. Those are stochastic processes that directly study the motion of the macroscopic population on the trait space, focussing on successful invasions and omitting the microscopic dynamics (see \cite{Krug21} for an overview). There are two sources of randomness in adaptive walks: A random fitness landscape and a random motion towards neighbours of higher fitness, according to some transition law.  Based on these, properties of interest are the distribution and accessibility of fitness maxima \cite{SchKr14,NoKr15,BeBruShi16,BeBruShi17}, as well as the time or path length to reach those maxima \cite{Orr03}. In adaptive flights, transitions are not just possible between neighbouring traits but from one local fitness maximum to another \cite{JainKrug05,JainKrug07,Jain07,KrugNeid2011}. This relates back to the limiting processes derived in this paper, where the population jumps between equilibrium states that are surrounded by valleys of traits of lower fitness.

A major difference between the models of adaptive walks/flights and adaptive dynamics is that the former assume a fitness landscape that is random but fixed in time, while in the latter case the fitness landscape is dynamic and depends on the current resident traits. As mentioned before, the notion of local fitness maxima can nevertheless be translated. Moreover, if equal competition between all traits is assumed in the adaptive dynamics model, the fitness landscape can again be regarded as fixed. We study this special case in a number of examples. Overall, the results of this paper can be seen as a validation of certain types of adaptive walks or flights, deriving their macroscopic dynamics from a microscopic, individual-based model.

The remainder of this paper is structured as follows: In Chapter \ref{sec:2_model_results}, we rigorously define the individual-based model of adaptive dynamics, for which we derive our limit theorems. We introduce key quantities, like the fitness of a trait, and recapitulate the most important results of \cite{CoKrSm21} that lead to a metastable state on the $\ln K$-time scale. Finally, we heuristically derive the limit behaviour on longer time scales and present the formal convergence results, starting with a single metastable transition in Section \ref{sec:2.2_first_result} and treating the full jump process in Section \ref{sec:2.3_second_result}. Chapter \ref{sec:3_examples} is devoted to the discussion of a number of examples that highlight different aspects of the complicated limiting dynamics in an easy set up. The proofs of the main results of this paper can be found in Chapter \ref{sec:4_proofs}. A combinatorial result on excursions of subcritical birth death processes and the complete version of the results from \cite{CoKrSm21} are stated in Appendix \ref{appendix}, for the convenience of the reader.

%%%%%%%%%%%%%%%%%%%%%%%%%%%%%%%%%%%%%%%%%%%%%%%%%%%%%%%%%%%%%%%%%%%%%

\section{Model and Main Results}
\label{sec:2_model_results}
In this chapter we introduce the individual-based model of adaptive dynamics and develop the main results of this paper. After a rigorous definition of the population process and its driving parameters we give a short overview of the behaviour on the time scales of order $1$ and $\ln K$ in Section \ref{sec:prevresults}. Moreover, in this section we derive the key quantities that lead us to the definition of an evolutionary stable condition.
Our main results on the transition out of an ESC are stated in Section \ref{sec:2.2_first_result} and we give a heuristic explanation there. Finally, Section \ref{sec:2.3_second_result} is devoted to our results on multi-scale jump chains and the convergence of the population process. For the convenience of the reader, we provide a preview of the different time scales and the main results of this paper at the end of Section \ref{sec:2.1_model}.

\subsection{Individual-based model %and important quantities
}
\label{sec:2.1_model}
To study the evolution of a heterogeneous population, we consider a classical stochastic individual-based model of adaptive dynamics. Each individual of our haploid population is characterised by its trait, which can be interpreted as its geno- or phenotype. Note that we assume a one to one correspondence between trait and physical properties. In this paper we consider a finite trait space that is given by a (possibly directed) graph $G=(V,E)$. Here, the set of vertices $V$ represents the possible traits that individuals can obtain. The set of edges $E$ marks the possibility of mutation between traits.

To each trait we associate a number of parameters that describe the dynamics of the system.
For $v,w\in V$ and $K\in\N$, denote by
\begin{itemize}
	\item $b(v)\in\R_+$, the \textit{birth rate} of an individual of trait $v$,
	\item $d(v)\in\R_+$, the \textit{(natural) death rate} of an individual of trait $v$,
	\item $c^K(v,w)=c(v,w)/K\in\R_+$, the \textit{competition} imposed by an individual of trait $w$  onto an individual of trait $v$,
	\item $\mu_K\in[0,1]$, the \textit{probability of mutation} at a birth event,
	\item $m(v,\cdot)\in\MM_p(V)$, the \textit{law of the trait of a mutant} offspring produced by an individual of trait $v$.
\end{itemize}
Here, $\MM_p(V)$ denotes the set of probability measures on $V$. The parameter $K$ scales the competitive interaction between individuals. It is called \textit{carrying capacity} and can be interpreted as the environment's capacity to support life, e.g.\ through the supply of nutrients or space. The way in which the mutation probability $\mu_K$ may depend on $K$ is discussed below.

To ensure a limited population size and to establish the connection between the possibility of mutation and the edges of our trait graph, we make the following assumptions on our parameters.
\begin{assumption}\label{Ass_selfcomp}
	For every $v\in V$, $c(v,v)>0$. Moreover, $m(v,v)=0$, for all $v\in V$, and $(v,w)\in E$ if and only if $m(v,w)>0$. 
\end{assumption}

The evolution of the population over time is described by the Markov process $N^K$ with values in $\mathbb{D}(\R_+,\N^V)$. $N^K_v(t)$ denotes the number of individuals of trait $v\in V$ that are alive at time $t\geq 0$. The process is characterised by its infinitesimal generator:
\begin{align}\label{Generator}
	\LL^K\phi(N)=&\sum_{v\in V}(\phi(N+\delta_v)-\phi(N))\left(N_vb(v)(1-\mu_K)+\sum_{w\in V}N_wb(w)\mu_Km(w,v)\right)\notag\\
	&+\sum_{v\in V}(\phi(N-\delta_v)-\phi(N))N_v\left(d(v)+\sum_{w\in V}c^K(v,w)N_w\right),
\end{align}
where $\phi:\N^V\to\R$ is measurable and bounded and $\delta_v$ denotes the unit vector at $v\in V$. The process can be constructed algorithmicly following a Gillespie algorithm \cite{Gill76}. Alternatively the process can be represented via Poisson measures (see \cite{FoMe04}), a representation that is used in the proofs of our results. Throughout this paper, we assume that all processes $N^K$, $K\in\N$, are defined on a common probability space. We give an example of a joint construction in the proof of Lemma \ref{lem:equilibriumSize}. However, we emphasize that we do not assume a specific dependence or independence between the different process in order for our results to hold true.

We want to study the typical behaviour of this process for large populations and moderately rare mutations. We do not have a fixed population size. However, due to our scaling of $c^K(v,w)$, the equilibrium size the population is always of order $K$. We therefore consider the limit of the processes $(N^K/K,K\in\N)$ as $K\to\infty$ and $\mu_K\to0$ simultaneously in this paper.

\textbf{Outlook:} In the following sections, we develop the theory to describe the systems behaviour on various time scales. Since the description of each increasing time scale builds on the behaviour on previous shorter time scales, we go through these step by step, introducing the relevant notation as well as previous and new results along the way. To give the reader some orientation, we provide a brief overview of the time scales and preview the main results:
\begin{itemize}
	\item During times of order 1, the limiting rescaled stochastic process can be approximated by the solution of deterministic differential equations of Lotka-Volterra type.  These describe how the larger subpopulations attain an equilibrium state (if existent). Since we consider the regime of $\mu_K\to0$, mutations cannot be observed on this time scale.
	\item For moderately rare mutations $\mu_K=K^{-1/\alpha}$, mutations occur on the time scale $1/K\mu_K$ and mutant subpopulations grow from a single individual to a size of order $K$ on the time scale $\ln K\gg1/K\mu_K$. The limiting dynamics on the $\ln K$-time scale have been described in \cite{CoKrSm21}. We provide the heuristics of this result in Section \ref{sec:prevresults} and give the precise statement in \ref{app:A2}. On this time scale, the system evolves until it reaches an equilibrium state, where there are no fit mutant traits of (graph-)distance at most $\alpha$ to the resident traits. This state is what we call an evolutionary stable condition (ESC).
	\item In Section \ref{sec:2.2_first_result}, we discuss how, on a more accelerated time scale $1/K\mu_K^L$ that corresponds to the distance $L>\alpha$ of the closest fit mutant, the process can escape an ESC. Our first result Theorem \ref{Thm:main} states that the time to produce a new fit mutant outside of the ESC is of order $1/K\mu_K^L$ and exponentially distributed with a rate that can be calculated precisely. It moreover states the probabilities to produce specific mutant types. Corollary \ref{Cor:main} deduces that the time to reach a new ESC has the same distribution as the time of leaving the old ESC and calculates transition probabilities to reach specific new ESCs. These single transitions between ESC states, which can be regarded as metastable transitions,  are used to define the (directed) metastability graph $\mathcal{G}_\ESC$ in Definition \ref{def:MetastabilityGraph}, in the beginning of Section \ref{sec:2.3_second_result}. It consists of subsets of $V$ that allow for an ESC and the possible transitions between them.
	\item Since the time scales on which transitions on the metastability graph occur depend on the distances $L$ between fit mutants and current resident traits, the corresponding jump chain (characterised in Corollary \ref{Cor:kProcess}) cannot be obtained as a limiting process on a single time scale. Instead, if we fix a time scale $1/K\mu_K^L$, only transitions of this precise distance $L$ are visible in the limit of $N^K/K$ as $K\to\infty$. Shorter jumps occur immediately and longer jumps cannot be observed. To describe these dynamics, we introduce an $L$-scale graph $\mathcal{G}^L$, consisting of all ESCs that are not left immediately on the time scale $1/K\mu_K^L$ and characterize the limiting jump process on this graph in Theorem \ref{Thm:Lscale}.
\end{itemize}

\subsection{Short-term dynamics and frequent mutations}\label{sec:prevresults}

A law of large numbers result by \cite{EtKu86} states that, for $\mu_K\equiv0$, the rescaled processes $N^K/K$ converge to the solution of a system of Lotka-Volterra equations. The study of these equations is central to determine the short term evolution, i.e.\ the evolution on a finite time scale, of the process $N^K$.

\begin{definition}[Lotka-Volterra system, equilibrium states, invasion fitness]
	For a subset $\v\subset V$ we denote by $LVS(\v)$ the system of \textit{Lotka-Volterra equations} given by
	\begin{align}\label{LV_system}
		\frac{\d}{\d t}{n}_v(t)= \left( b(v) - d(v) - \sum_{w \in \v} c(v,w) n_w(t) \right) n_v(t), \quad v\in\v,\ t\geq0. 
	\end{align}
	By $LVE(\v)$, we denote the set of all \textit{equilibrium states} $\bar{n}\in\R_{\geq0}^\v$ such that
	\begin{align}
		\left( b(v) - d(v) - \sum_{w \in \v} c(v,w) \bar{n}_w \right) \bar{n}_v=0, \quad v\in\v,
	\end{align}
	and by $LVE_+(\v):=LVE(\v)\cap\R_{>0}^\v$ the subset of \textit{positive equilibrium states}. If $LVE_+(\v)$ consists of a single asymptotically stable element, we denote it by $\bar{n}(\v)$ and call it \textit{coexistence equilibrium}.\\
	For a trait $w\in V$ and coexistence equilibrium $\bar{n}(\v)$, we denote by
	\begin{align}\label{eq:fitness_def}
		f(w,\v)=b(w)-d(w)-\sum_{v\in\v}c(w,v)\bar{n}_v(\v)
	\end{align}
	the \textit{invasion fitness} of $w$. For a given equilibrium $\bar{n}(\v)$, we call a trait $w$ \textit{fit} if $f(w,\v)>0$ and \textit{unfit} if $f(w,\v)<0$.
\end{definition}

Note that the invasion fitness $f(w,\v)$ describes the approximate growth rate of a small population of trait $w$ in a bulk population of coexisting traits $\v$, in the mutation-free system. To simplify notation for later purpose, in the case of monomorphic equilibria, i.e.\ $\v=\{v\}$, we write
\begin{align}
\bar{n}(v):=\bar{n}_v(\{v\})\quad\text{and}\quad f(w,v):=f(w,\{v\}).
\end{align}

Going back to the stochastic process $N^K$, it is of interest to study the logarithm of the population size as $K\to\infty$. Only subpopulations with a size of order $K$ are visible in the rescaled limit of $N^K/K$ and exponential growth of the absolute population size translates to linear growth of the $K$-exponent when studying a logarithmic time scale via $e^{t\ln K\cdot f}=K^{t\cdot f}$. This makes it easier to describe the limiting dynamics. We therefore define $\beta^K=(\beta^K_v)_{v\in V}$, where
\begin{align}
	\beta^K_v(t):=\frac{\ln(1+N^K_v(t))}{\ln K},
\end{align}
which is equivalent to $N^K_v(t)=K^{\beta^K_v(t)}-1$. Note that we add or subtract 1 here respectively to ensure that $\beta^K_v(t)=0$ if and only if $N^K_v(t)=0$. As $K\to\infty$, $\beta^K_v$ ranges between $0$ and $1$.

\begin{remark}
	In contrast to \cite{ChMeTr19,CoKrSm21}, we do not rescale the time by $\ln K$ in this definition of $\beta^K$ since we are studying a variety of different time scales.
\end{remark}

Based on this definition, we introduce the following subsets of traits.
\begin{definition}[macroscopic, microscopic, living and resident traits]
	\begin{enumerate}
		\item A trait ${v\in V}$ with exponent ${\beta^K_v}$ is called \emph{macroscopic} if ${\liminf_{K\to\infty}\beta^K_v=1}$.
		\item A trait that is not macroscopic is called \emph{microscopic}.
		\item The set of \emph{living traits} is the set $V^K_\liv:=\{v\in V : \beta^K_v>0\}$.
		\item A subset of traits $\v\subseteq V$ is called \textit{resident} if all $v\in\v$ are macroscopic and have a population size close to the coexistence equilibrium $\bar{n}(\v)$.
	\end{enumerate}
\end{definition}

\begin{remark}
	Note that these definitions are time dependent when considering an evolving population. The macroscopic traits change according to $\beta^K(t)$ and the varying subset of living traits is denoted by $V^K_\liv(t)$.
	Most of the time macroscopic and resident traits coincide. A non-resident macroscopic trait is either unfit and will shrink to an order lower than $K$ within a short time, or it is fit and will therefore induce a change in resident traits according to the short-term Lotka-Volterra dynamics.
\end{remark}

To study multi-step mutations we consider paths on the trait graph \mbox{$G=(V,E)$}.
\begin{definition}[paths and distances]
	We denote a \textit{(finite) path} on $G=(V,E)$ by\linebreak $\g=(\g_0,...,\g_\ell)$ such that $\g_i\in V$, $0\leq i\leq\ell$, and $(\g_i,\g_{i+1})\in E$, $0\leq i\leq\ell-1$. \\
	The \textit{length of a path} $\g=(\g_0,...,\g_\ell)$ is defined as $\abs{\g}=\ell$. We write $\g:\v\to \v'$ as a short notation for all paths $\gamma$ that connect $\v\subset V$ to $\v'\subset V$, i.e.\ that satisfy $\g_0\in\v$ and $\g_{\abs{\g}}\in\v'$.\\
	We introduce the graph distance between two vertices $v,w\in V$ as the length of the shortest connecting path
	\begin{align}
		d(v,w):=\min_{\g:v\to w}\abs{\g},
	\end{align}
	where the minimum over an empty set is taken to be $\infty$.
	For two subsets $\v,\v'\subset V$ we define
	\begin{align}
		d(\v,\v'):=\min_{v\in\v,v'\in\v'}d(v,v').
	\end{align}
\end{definition}

\begin{remark}
	Note that $d(v,w)$ is not a distance in the classical sense, as it may not be symmetric in the case of a directed graph.
\end{remark}

Along these paths $\g$, mutants can be produced. A macroscopic trait produces subpopulations of a size of order $K\mu_K$ of its neighbouring traits, which then produce subpopulations of a size of order $K\mu_K^2$ of the second order neighbours, and so on. These subpopulations, that are produced along a path $\g$, can survive as long as $K\mu_K^\ell\gg1$. This motivates the study of mutation probabilities $\mu_K=K^{-1/\alpha}$, $\alpha>0$, where mutants can survive within a radius $\alpha$ of the resident traits.

\begin{remark}
	We could also study mutation probabilities $\mu_K=f(K)K^{-1/\alpha}$ such that $\abs{\ln(f(K))}\in o(\ln K)$. This would not change the following results. However, we restrict ourselves to the case of $f(K)\equiv1$ to simplify notation.
\end{remark}

To avoid mutant subpopulations with a size of order $K^0=1$ and to ensure that non-resident traits are always either fit or unfit we make the following assumptions.

\begin{assumption}
	\label{Ass_alphafit}
	\begin{enumerate}[(i)]
		\item The mutation probability satisfies $\mu_K=K^{-1/\alpha}$ for some $\alpha \in \R_+\setminus \N$.
		\item For each $\v\subset V$ such that $LVE_+(\v)=\{\bar{n}(\v)\}$, it holds $f(w,\v)\neq0$, for all $w\notin\v$.
	\end{enumerate}

\end{assumption}

\begin{remark}
	Both of these assumptions are purely technical. The first one prevents the case where a fit mutant population of size of order 1 can die out due to stochastic fluctuations such that fixation in the population becomes random. The second one allows us to approximate non-resident subpopulations by branching processes that are either super- or subcritical, but not critical. Note that the second assumption is only required for subsets $\v$ that allow for a unique positive equilibrium state (i.e.\ such that $LVE_+(\v)$ contains exactly one element).
\end{remark}

Under these assumptions, the evolution of the population on the time scale $\ln K$ has been studied in \cite{CoKrSm21}. The authors give an algorithmic description of the limiting evolution of $\beta^K(t\ln K)$ as long as there always exists a unique asymptotically stable equilibrium of the Lotka-Volterra system \eqref{LV_system} involving all macroscopic traits. In the following, we give the heuristics of this description. For the precise result we refer to Section \ref{proof_logK}.

Roughly speaking, for a given set of resident traits $\v$ at their (coexistence) equilibrium $\bar{n}(\v)$, every living microscopic trait $w\in V_\liv$ can grow (or shrink) with rate at least $f(w,\v)$. This is due to the fact that the competitive interaction with all microscopic traits can be neglected in comparison with this rate. If there was no mutation (i.e.\ $\mu_K=0$), $f(w,\v)$ would be the exact growth rate of $w$. However, due to incoming mutants from neighbouring traits, the population size of $w$ is also at least as big as a $\mu_K$-fraction of the population sizes of its (incoming) neighbours. Since we only consider the order of the population size $\beta^K_w$, the largest of these influences dominates the asymptotics and a sum of population sizes (coming from different mutation sources) yields a maximum in the exponent. Overall, we obtain the relation
\begin{align}
	\beta^K_w(t\ln K)\approx\left(\beta^K_w(0)+tf(w,\v)\right)\lor\max_{u\in V: d(u,w)=1} \left(\beta^K_u(t\ln K)-\frac{1}{\alpha}\right).
\end{align}

Iterating this argument for traits at increasing distance to $w$ yields that, as long as the resident traits remain unchanged (i.e.\ traits $\v$ stay close to their equilibrium $\bar{n}(\v$) and no new traits become macroscopic), $\beta^K(t \ln K)$ converges to $\beta(t)$ such that
\begin{align}\label{logK_growth}
	\beta_w(t)=\max_{u\in V}\left[\beta_u(0)+(t-t_u)f(u,\v)-\frac{d(u,w)}{\alpha}\right]_+.
\end{align}
Here,
\begin{align}
t_u:=\begin{cases}
\inf\left\{s\geq 0:\exists\ u'\in V:\b_{u'}(s)=\tfrac{1}{\alpha},(u',u)\in E\right\}&\text{if }\b_u(0)=0,\\
0&\text{if }\b_u(0)>0.
\end{cases}
\end{align}

Once a former microscopic trait $w^*$ becomes macroscopic, the population sizes of $\v\cup w^*$ follow the Lotka-Volterra dynamics of \eqref{LV_system} to reach a new equilibrium associated to the resident traits $\v'\subset\v\cup w^*$ within a time of order 1 (if such a new unique equilibrium does not exist, or in a number of other technical special cases, the algorithm terminates as described in Section \ref{proof_logK}). During this phase, the orders of population sizes $\beta_w$ do not change significantly. After the change of resident traits, the population sizes again follow \eqref{logK_growth}, now with the changed fitnesses $f(u,\v')$.

This algorithmic description yields a series of successive resident traits. The macroscopically visible evolution stops as soon as an equilibrium $\v$ is reached such that $f(w,\v)<0$ for all $w\in V_\liv\backslash \v$. All traits $w\in V$ such that $d(\v,w)<\alpha$ stay alive due to incoming mutations but all other traits eventually go extinct according to \eqref{logK_growth} on the $\ln K$-time scale.

This observation leads us to the following definitions (visualised in Figure \ref{fig:TraitGraph}).
\begin{definition}[mutation spreading neighbourhood]
	For a subset $\v\subset V$, we denote by $V_\alpha(\v):=\{w\in V: d(\v,w)<\alpha\}$ the \textit{mutation spreading neighbourhood} of $\v$. The traits at the boundary of $V_\alpha$ are denoted by $\partial V_\alpha(\v):=\{w\in V: d(\v,w)=\lalpha\}$.
\end{definition}	

\begin{definition}[(asymptotic) evolutionary stable condition]
	\label{def:ESC}
	\begin{itemize}
		\item[(i)] A subset $\v\subset V$ and (orders of) population sizes $\beta$ are called an \textit{evolutionary stable condition (ESC)} if the traits $\v$ can coexist at equilibrium $\bar{n}(\v)$,
		\begin{align}\label{ESCfitness}
			f(w,\v)<0,\ \forall w\in V_\alpha(\v)\backslash\v,
		\end{align}
		and
		\begin{align}\label{ESCsize}
			\beta_w=\left(1-\frac{d(\v,w)}{\alpha}\right)_+,\ \forall\ w\in V.
		\end{align}
		\item[(ii)] A subset $\v\subset V$ and population sizes $(\beta^K)_{K\geq0}$ are called an \textit{asymptotic evolutionary stable condition} if the traits $\v$ can coexist at equilibrium $\bar{n}(\v)$, \eqref{ESCfitness} is satisfied,
		\begin{align}
	    		\abs{\b_w^K-(1-d(\v,w)/\alpha)}\in O\left(\frac{1}{\ln K}\right), \ \forall w\in V_\a(\v),
	    	\end{align}
	    	and there exists a $K_0<\infty$ such that $\beta^K_w=0$, for all $K>K_0$ and \mbox{$w\in V\backslash V_\alpha(\v)$}.
	\end{itemize}
\end{definition}

\begin{remark}
	\begin{itemize}
	\item[(i)] Note that \eqref{ESCfitness} is only a necessary condition for a subset $\v\subset V$ to be able to attain an ESC during the evolution of a population. \eqref{ESCsize} are the orders of population sizes that unfit traits stabilise at purely due to (multi-step) mutations from $\v$. \eqref{ESCfitness} guarantees that these will be reached for $w\in V_\alpha(\v)$. To attain an ESC $(\v,\beta)$, in addition all other traits $w\in V_\liv(\tau_\v)$, that are alive at the time $\tau_\v$ when the new equilibrium $\bar{n}(\v)$ is reached, have to satisfy $f(w,\v)<0$. If this is the case, all traits outside of $V_\alpha(\v)$ will die out within a $\ln K$ time and \eqref{ESCsize} will be reached. Otherwise, if there is a $w\in V_\liv(\tau_\v)\backslash V_\alpha(\v)$ such that $f(w,\v)>0$ (the case $f(w,\v)=0$ is excluded by Assumption \eqref{Ass_alphafit}), its subpopulation is able to grow, will not die out, and hence not satisfy \eqref{ESCsize}. The characterization of ESCs is therefore highly dependent on the state of the whole system.
	\item[(ii)] Note that the definition of an asymptotic ESC forces the population process to be in an ESC up to a multiplicative error of order one. That is
		\begin{align}
			N_w^K=(K^{(1-d(\v,w)/\a)_+}-1)\times \mathcal{O}(1).
		\end{align}
	The reason for introducing this error is that, for finite $K$, $N^K_w$ might never reach exactly $K^{(1-d(\v,w)/\a)_+}$. This is for example the case if $\bar{n}_v(\v)<1$ for some $v\in\v$.
	\end{itemize}
\end{remark}

By definition, an evolutionary stable condition is surrounded by unfit traits, at least within an $\alpha$-radius. This form of a fitness landscape is referred to as a \textit{fitness valley} and has been studied in a special case in \cite{BoCoSm19}. Based on this, we introduce a measure for the stability of a coexistence equilibrium, connected to the width of the surrounding fitness valley.

\begin{definition}[Stability degree]
		For a subset $\v\subseteq V$ we define its \emph{stability degree} $L(\v)$ by
		\begin{align}
			L(\v):=\begin{cases}\min_{w\in V: f(w,\v)>0}d(\v,w)&\text{if $\v$ can coexist,}\\
			0&\text{else}.\end{cases}
		\end{align}
\end{definition}

\begin{remark}\label{rem:ESCstability}
A subset $\v$ associated to an ESC satisfies $L(\v)>\a$ by definition. The evolution of the population process reaches a final state, independent of the time scale, once the resident traits satisfy $L(\v)=\infty$, i.e.\ there are no fit traits anymore.
\end{remark}

%%%%%%%%%%%%%%%%%%%%%%%%%%%%%%%%%%%%%%%%%%%%%%%%%%%%%%%%%%%%%%%%%%%%%%%%%%%%%%%%%%%%%%%%%

\subsection{Transitioning out of an ESC and first convergence result}
\label{sec:2.2_first_result}

Once an ESC is obtained, there is no further evolution on the $\ln K$-time scale. However, as long as there is a fitter trait that is connected to the resident traits, i.e.\ that can be reached along a finite path in $G$, we can witness metastable transitions on an even more accelerated time scale. On this time scale, under certain assumptions on the $\ln K$-dynamics, we observe a direct transition from one ESC to another.

In the following, we consider one of these transitions for an arbitrary initial asymptotic ESC. We split the transition into two phases: In the first phase, a new fit mutant beyond the fitness valley fixates in the population within a time of order $1/K\mu_K^{L(\v)}$.  In the second phase, a new ESC is obtained, starting with these new initial conditions, which takes a time of order $\ln K$. We assume that $\v$ and $(\b^K(0))_{K\geq0}$ are an asymptotic ESC.  We could also consider more general initial conditions that lead to an ESC within finitely many steps of the $\ln K$-algorithm in \cite{CoKrSm21}, see Remark \ref{rem:initialESC}. For the sake of a simpler notation, we stick with the assumption of starting in an (asymptotic) ESC here.

To consider the first phase of the transition, we introduce the set
\begin{align}
	V_\mut(\v)&:=\argmin_{w\in V: f(w,\v)>0}d(\v,w)=\{w\in V:f(w,\v)>0,d(\v,w)=L(\v)\}.
\end{align}
This consists of all fit mutant traits that are closest to $\v$ (visualised in Figure \ref{fig:TraitGraph}).

Note that $V_\mut(\v)\cap V_\alpha(\v)=\emptyset$ by the definition of an ESC. It turns out that the traits $V_\mut(\v)$ are the only ones that need to be considered for a crossing of the fitness valley since one of them will be the first new trait to fixate in the equilibrium population. If $V_\mut(\v)=\emptyset$, i.e.\ $L(\v)=\infty$, there is no fitter trait connected to the resident traits and the equilibrium associated to $\v$ is the final state of the population.

For $L(\v)<\infty$, we define the stopping time
\begin{align}
	T^K_\fix:=\inf\left\{t\geq 0:\exists\ w\in V\backslash V_\a(\v):\beta^K_w(t)\geq\frac{1}{\alpha}\right\},
\end{align}
the first time when a new trait reaches a size of order $K^{1/\alpha}$, can thus produce neighbouring mutants within a time of order 1 and influence the subpopulations of other traits. 

\begin{remark}
		Note that the name $T^K_\fix$ might be a little misleading at first glance. Generally, we speak of the fixation of a trait within a population as the event that the subpopulation corresponding to this trait does not go extinct (due to random fluctuations or negative fitness), as long as the fitness landscape stays unchanged. As this event is determined by the future progression of the population, there is no precise time point to pin it to. In particular, whether a trait fixates or goes extinct is not foreseeable at the time point when the first individual of this trait arises. Therefore, we choose instead the time point when the subpopulation has reached a size that guarantees non-extinction with probability 1, asymptotically as $K\to\infty$. We could choose a much smaller size than $K^{1/\alpha}$ for this, however, this will not influence the event of fixation and only change the stopping time by a time of order $\ln K$, which is negligible compared to the much longer time scale on which mutants arise. We thus pick the first time when mutants can influence the population size of other traits.
\end{remark}

\begin{figure}[h]
	\centering
	\includegraphics[width=.95\textwidth]{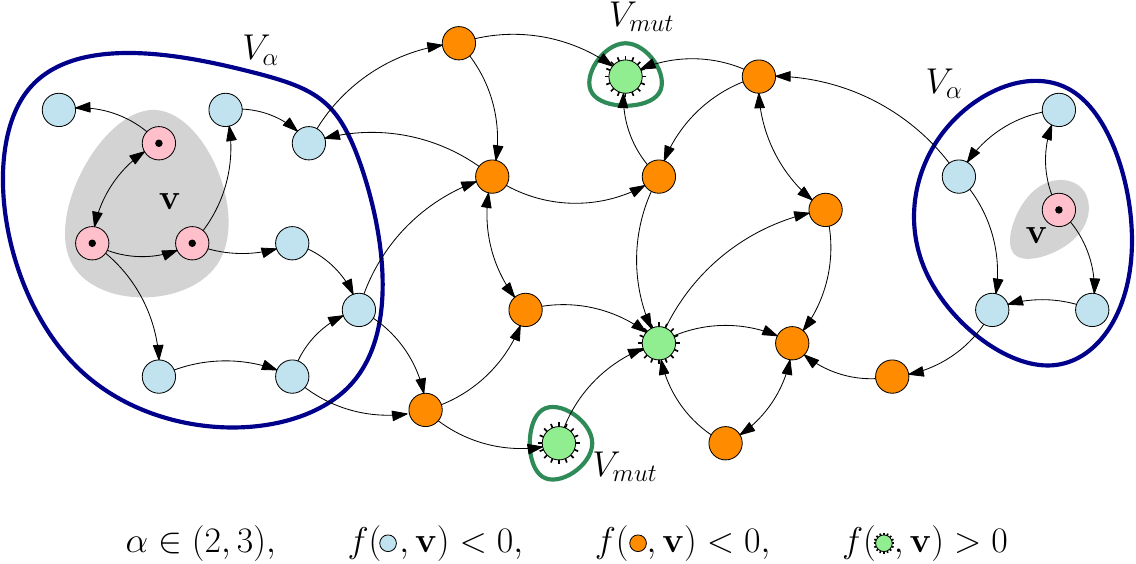}
	\caption{Trait graph $G=(V,E)$ with an ESC associated to the non-connected set of spotted resident traits $\v$. Unfit traits inside the mutation spreading neighbourhood $V_\a(\v)$ are marked light blue while those outside are dark orange. In this case, the stability degree is $L(\v)=4$. Therefore, only the two nearest of the three fit, spiked green traits make up the set of mutant candidates $V_\mut(\v)$.}
	\label{fig:TraitGraph}
\end{figure}

Our first result describes the limiting distribution of this stopping time $T^K_\fix$.

For a path $\g:\v\to V_\mut(\v)$ such that $\abs{\g}=L(\v)$, the rate at which a $w=\g_{L(\v)}$ mutant population arises along this path $\g$ and fixates can be derived as the product of several factors. The rate at which the first trait in $\g$ outside of $V_\alpha(\v)$ arises can be calculated in terms of the equilibrium population sizes of the traits in $V_\alpha(\v)$ (see Section \ref{sec:4.1_proof_equilsize}). This rate then has to be multiplied by the probabilities that all of the following unfit traits on the path $\g$ produce mutants of the correct trait before extinction, i.e.\ during small subcritical excursions. This yields the rate at which single  mutants of trait $w$ arise, which finally has to be multiplied by their probability of fixating in the population, i.e.\ of non-extinction.

In order to calculate the probability of mutation during a subcritical excursion, we need to introduce some notation. For a subset $\v\subset V$ and a trait $w\in V$ we define
\begin{align}
\rho(w,\v):=\frac{b(w)}{b(w)+d(w)+\sum_{v\in\v}c(w,v)\bar{n}_v(\v)}\label{def:rho},
\end{align}
which is connected to the probability of a birth event in the branching process approximating the growth of a mutant $w$ in a bulk population of coexisting traits $\v$. Moreover, we let
\begin{align}
\lambda(\rho):=\sum_{\ell=1}^\infty\frac{(2\ell)!}{(\ell-1)!(\ell+1)!}\rho^\ell(1-\rho)^{\ell+1},\label{def:lambda}
\end{align}
which is the expected number of birth events before extinction in a subcritical birth death process with birth probability $\rho$ (related to the expected number of positive jumps in a simple random walk on $\N$ before hitting 0, as explained in the proof of this result in Section \ref{appendix}, Lemma \ref{lem:offspringOfAnExcursion}). Note that this sum is finite, as long as $\rho\in[0,1]\backslash(1/2)$. This can be seen by bounding $\binom{2\ell}{\ell-1}<\binom{2\ell}{\ell}\leq(4)^\ell$ and $\rho(1-\rho)<1/4$.

With these definitions, the overall rate of mutants of trait $w\in V_\mut(\v)$ arising along path $\g$ and fixating in the population is approximately equal to $R(\v,\g)\mu_K^{L(\v)}$, where
\begin{align}
\label{eq:PathRate}
R(\v,\g):=&\bar{n}_{\g_0}(\v)\left(\prod_{i=1}^{\lalpha}\frac{b(\g_{i-1})m(\g_{i-1},\g_i)}{\abs{f(\g_i,\v)}}\right)
b(\gamma_{\lalpha})m(\gamma_{\lalpha},\gamma_{\lalpha+1})\notag\\
&\cdot\left(\prod_{j=\lalpha+1}^{L(\v)-1}\lambda(\rho(\g_j,\v))m(\g_j,\g_{j+1})\right)\cdot\frac{f(\g_{L(\v)},\v)}{b(\g_{L(\v)})}.
\end{align}
Here, the first line is the rate at which the first trait in $\g$ outside of $V_\a(\v)$ arises, which is related to the equilibrium size of trait $\g_{\lalpha}$. The first factor in the second line is the probability of producing consecutive mutants during subcritical excursions and the last factor is the fixation probability of trait $w=\g_{L(\v)}$. Note that, as $b(\g_{L(\v)})$ increases, so does $f(\g_{L(\v)},\v)$ (cf.\ \eqref{eq:fitness_def}), and hence this fixation probability is in fact increasing in the birth rate $b(\g_{L(\v)})$. 

The total rate at which a mutant population of trait $w\in V_\mut(\v)$ arises and fixates collects all shortest paths that end in $w$ and is approximately equal to $R(\v,w)\mu_K^{L(\v)}$, where
\begin{align}\label{eq:TraitRateFix}
R(\v,w):=\sum_{\substack{\gamma:\v\to w\\\abs{\g}=L(\v)}}R(\v,\g).
\end{align}

Finally, the total rate at which any mutant population of a trait in $V_\mut(\v)$ arises and fixates, i.e.\ the rate at which the population exits the ESC associated to $\v$, is approximately equal to $R(\v)\mu_K^{L(\v)}$, where
\begin{align}\label{eq:ExitRate}
R(\v):=\sum_{w\in V_\mut(\v)}R(\v,w).
\end{align}
The probability that this population is of trait $w\in V_\mut(\v)$ is proportional to the rate $R(\v,w)$.

With these heuristics, we can now state the first main result of this paper.
\begin{theorem}\label{Thm:main}
	Let $G=(V,E)$ be a finite graph. Suppose that Assumption \ref{Ass_selfcomp} and \ref{Ass_alphafit} are satisfied and consider the model defined by \eqref{Generator} with $\mu_K=K^{-1/\alpha}$. Assume that $\v\subset V$ and $(\beta^K(0))_{K\geq0}$ are an asymptotic ESC.
	Then there exist constants $\eps_0>0$ and $0<c<\infty$ such that, for all $0<\eps<\eps_0$, there exist exponential random variables $E_+(\eps)$ and $E_-(\eps)$ with parameters $R(\v)(1+c\eps)$ and $R(\v)(1-c\eps)$, such that
	\begin{align}
	\liminf_{K\to\infty}\P(E_-(\eps)\leq T^K_\fix K\mu_K^{L(\v)}\leq E_+(\eps))\geq 1-c\eps.
	\end{align}
	Moreover, for all $w\in V$, the probability of $w$ being the trait to trigger $T^K_\fix$ is
	\begin{align}\label{FixatingTrait}
	\lim_{K\to\infty}\P\left(\beta^K_w(T^K_\fix)=1/\alpha\right)=\begin{cases}R(\v,w)/R(\v)&\text{if }w\in V_\mut(\v),\\0&\text{else.}\end{cases}
	\end{align}
\end{theorem}

\begin{remark}
	Note that traits in $w\in V_\alpha(\v)$ do not attain $\beta^K_w=1/\a$ before $T^K_\fix$ due to the assumption that $\alpha\notin\N$. Therefore the probability in \eqref{FixatingTrait} is zero for such traits.
\end{remark}

Once some $w\in V_\mut(\v)$ has reached $\beta^K_w\geq1/\alpha$, the $\ln K$-dynamics evolve as described in \cite{CoKrSm21}, initiated with $\beta^K_w=1/\alpha$ and $\beta^K_u=(1-d(\v,u)/\alpha)_+$, for $u\in V\backslash w$. These dynamics are deterministic and in case they do not terminate early and if they lead to a new ESC, we denote the associated set of resident traits by $\v_\ESC (\v,w)$.

Observe that there is no general formula to express $\v_\ESC (\v,w)$ in terms of $\v$ and $w$ and the parameters of the system. An interesting case is illustrated in Example \ref{ex:5_Selfconn}.

Under the assumption that all traits $w\in V_\mut(\v)$ lead to asymptotic ESCs $\v_\ESC (\v,w)$, we define the stopping time at which one of these asymptotic ESCs is obtained by
\begin{align}
	T^K_\ESC :=\inf\Bigg\{&t\geq T^K_\fix:\exists\ w\in V_\mut(\v):\\
		&\forall u\in V_\alpha(\v_\ESC (\v,w)):\abs{\beta^K_u(t)-\left(1-\frac{d(\v_\ESC (\v,w),u)}{\alpha}\right)}<\eps_K,\notag\\
		&\forall u\notin V_\alpha(\v_\ESC (\v,w)):\beta^K_u(t)=0\Bigg\},
\end{align}
	where we pick $\eps_K=C/\ln K$ for a large enough $0<C<\infty$. Then this definition is precisely in line with the definition of an asymptotic ESC.
\begin{remark}
The minimal necessary $C$ can be made precise using the prefactors of the population sizes in equilibrium, calculated in Lemma \ref{lem:equilibriumSize}. We refrain from doing so here as it is notationally very heavy and does not provide any deeper insight.
	%e.g.\ $C=\abs{\ln(2\max_{\v\subset V,v\in\v}\bar{n}_v(\v))}\lor\abs{\ln(\min_{\v\subset V,v\in\v}\bar{n}_v(\v))/2}$
\end{remark}

Since the time $T^K_\ESC -T^K_\fix$ is of order $\ln K$, the asymptotics for $T^K_\fix$ translate to $T^K_\ESC$. Moreover, the transition probabilities from one ESC to another can be expressed in terms of the probabilities of traits $w\in V_\mut(\v)$ fixating in the population. For $\w\subset V$ we define
\begin{align}\label{eq:TransitionRate}
	p(\v,\w):=\sum_{\substack{w\in V_\mut(\v):\\\v_\ESC (\v,w)=\w}}\frac{R(\v,w)}{R(\v)}.
\end{align}
Example \ref{ex:0_combined} treats a case where this probability is indeed the sum over multiple mutant candidates $w$.

We can now state the result on transitions between ESCs as a direct corollary of Theorem \ref{Thm:main}.
\begin{corollary}\label{Cor:main}
Suppose the same assumptions as in Theorem \ref{Thm:main} are satisfied. Moreover, assume that, for every $w\in V_\mut(\v)$, the algorithmic description of the $\ln K$-dynamics in Section \ref{proof_logK}, initiated with
	\begin{align}\label{eq:afterFix}
	\beta_u(0)=\begin{cases}\frac{1}{\alpha}&\text{if }u=w\\
	\left(1-\frac{d(\v,u)}{\alpha}\right)_+ &\text{else}\end{cases},
    \end{align}
    does not stop early due to one of its termination criteria and reaches an ESC associated to some traits $\v_\ESC (\v,w)$ after finitely many steps. Then, $T^K_\ESC -T^K_\fix\in O(\ln K)$ and therefore, with the same constants $\eps_0$ and $c$ and with the same random variables $E_+(\eps)$ and $E_-(\eps)$ as in Theorem \ref{Thm:main}, 
    \begin{align}
		\liminf_{K\to\infty}\P(E_-(\eps)\leq T^K_\ESC K\mu_K^{L(\v)}\leq E_+(\eps))\geq 1-c\eps.
	\end{align}
	Moreover, for all $\w\subset V$,
	\begin{align}
		\lim_{K\to\infty}\P(\{u\in V:\beta^K_u(T^K_\ESC )> 1-\eps_K\}=\w)=p(\v,\w).
	\end{align}
\end{corollary}

\begin{remark}\label{rem:initialESC}
	\begin{itemize}
		\item[(i)] Note that Theorem \ref{Thm:main} and Corollary \ref{Cor:main} only consider a specific transition from the ESC associated to some $\v$ to another ESC. The constants $\eps_0$ and $c$ can however be chosen uniformly for all ESCs by reason of the finite trait graph.
		\item[(ii)] Both results assume that the system starts out in an asymptotic ESC. These are the natural initial conditions, particularly when a first transition between asymptotic ESCs has already occurred. We could however allow for more general initial conditions, as long as they lead to an asymptotic ESC within finitely many steps of the $\ln K$-algorithm.
		%We could also choose more general initial conditions, as long as they lead to an asymptotic ESC within finitely many steps of the $\ln K$-algorithm. This adds another $\ln K$ time to $T^K_\fix$ and $T^K_\ESC$ respectively, which is negligible in the limit. Since the limit of the $\ln K$-dynamics is deterministic, each admissible initial condition leads to a unique ESC associated to some traits $\v\subset V$ and the statements of the theorem and corollary stay valid as before.
	\end{itemize}
\end{remark}

\subsection{Multi-scale jump chain and limiting Markov jump processes}
\label{sec:2.3_second_result}

Building on the previous description of a single transition step from one ESC to another,  we now want to describe the multi-step transitions between ESCs as a jump chain $(v^{(k)})_{k\geq 0}$ on a meta-graph. We first introduce the underlying \emph{metastability graph} $\GG_\ESC$, consisting of all sufficiently stable macroscopic equilibrium configurations, and then describe the dynamics of the jump chain. Finally, we give a convergence result that derives different Markov jump processes, depending on the chosen time scale.

\begin{definition}[Metastability graph]
	\label{def:MetastabilityGraph}
	As vertices for the general metastability graph \linebreak $\mathcal{G}_\ESC =\left(\mathcal{V}_\ESC ,\mathcal{E}_\ESC \right)$ we take all sets of resident traits that correspond to an ESC, i.e.\ that have stability degree strictly bigger than $\alpha$, and edges represent possible transitions to other ESCs. More precisely,
	\begin{align}
	\VV_\ESC :=&\dset{\v\subseteq V : L(\v)>\alpha},\\
	\EE_\ESC :=&\dset{(\v,\w):\exists w\in V_\mut(\v) \text{\ s.t.\ } \w=\v_\ESC (\v,w)}.
	\end{align}
\end{definition}
Recall that $\v_\ESC (\v,w)$ stands for the resident traits associated to the new ESC that is attained at the end of the $\ln K$-algorithm being started with resident set $\v$ and invading mutant $w\in V_\mut(\v)$. We already assigned to each vertex $\v\in\VV_\ESC$ the exit rate $R(\v)$ in \eqref{eq:ExitRate} and to each edge $(\v,\w)\in\EE_\ESC$ the transition probability $p(\v,\w)$ in \eqref{eq:TransitionRate}.

Using Corollary \ref{Cor:main}, we can now work out inductively the multi-scale jump chain $(\v^{(k)})_{k\geq0}$ on $\GG_\ESC$. To this end, let $\v^{(0)}\in\mathcal{V}_\ESC$ be the resident traits of the initial ESC that the process starts in and set $T^{(0,K)}_\ESC :=0$. We describe the $k^\text{th}$ transition, for $k\geq 1$, conditioned on the knowledge of $\v^{(k-1)}$. We denote the set of closest fit mutant traits by $V_\mut^{(k)}=V_\mut(\v^{(k-1)})$, the width of the next fitness valley to cross by $L^{(k)}=L(\v^{(k-1)})$, and the exit rate by $R^{(k)}=R(\v^{(k-1)})$. Moreover, we
keep track of the time when the first mutant population fixates and when the next ESC is reached by introducing the stopping times
\begin{align}
T^{(k,K)}_\fix:=\inf\Bigg\{&t\geq T^{(k-1,K)}_\ESC :\exists\ w\in V\backslash V_\a(\v^{(k-1)}):\beta^K_w(t)\geq\frac{1}{\alpha}\Bigg\},\\
T^{(k,K)}_\ESC :=\inf\Bigg\{&t\geq T^{(k,K)}_\fix:\exists\ w\in V_\mut^{(k)}:\nonumber\\
&\forall u\in V_\alpha(\v_\ESC (\v^{(k-1)},w)):\nonumber\\
&\qquad\qquad\abs{\beta^K_u(t)-\left(1-\frac{d(\v_\ESC (\v^{(k-1)},w),u)}{\alpha}\right)}<\eps_K,\notag\\
&\forall u\notin V_\alpha(\v_\ESC (\v^{(k-1)},w)):\beta^K_u(t)=0\Bigg\}.
\end{align}

With this notation, we can now state the result on the $k^\text{th}$ transition of the multi-scale jump chain.

\begin{corollary}
	\label{Cor:kProcess}
	Assume that we constructed the process up to time $T^{(k-1,K)}_\ESC$, when the ESC associated to $\v^{(k-1)}$ is obtained, and suppose the same assumptions as in Theorem \ref{Thm:main} are satisfied. Moreover assume that the  $\ln K$-dynamics behave as in Corollary \ref{Cor:main}, for every $w\in V_\mut^{(k)}$. Then there exist constants $\eps_0>0$ and $0<c<\infty$ such that, for all $0<\eps<\eps_0$, there are exponential distributed random variables $E^{(k)}_+(\eps)$ and $E^{(k)}_-(\eps)$ with parameters $R^{(k)}_\pm(\eps):=R^{(k)}(1\pm c\eps)$ such that
	\begin{align}
	\liminf_{K\to\infty}\P(E^{(k)}_-(\eps)\leq (T_\ESC^{(k,K)}-T_\ESC^{(k-1,K)})K\mu_K^{L^{(k)}}\leq E^{(k)}_+(\eps)\vert\v^{(k-1)})\geq 1-c\eps.
	\end{align}
	Moreover, for all $\w\subset V$,
	\begin{align}
	\lim_{K\to\infty}\P(\{v\in V:\beta^K_v(T^{(k,K)}_\ESC )> 1-\eps_K\}=\w\vert\v^{(k-1)})=p(\v^{(k-1)},\w).
	\end{align}
\end{corollary}

The preceding corollary allows us to construct a limiting random jump chain $(\v^{(k)})_{k\geq 0}$ on the metastability graph $\GG_\ESC$. To be precise, given the current state $\v^{(k-1)}$, the next ESC $\v^{(k)}$ is taken at random from $\VV_\ESC$ with probability distribution $p(\v^{(k-1)},\cdot)$. However, the jumps take place on varying time scales of type $1/K\mu_K^{L^{(k)}}$.
The construction is valid until an ESC is obtained such that some mutant $w\in V_\mut^{(k)}$ does not induce a unique new ESC, following the deterministic $\ln K$-dynamics.
A visualisation of the metastability graph including a particular realisation of the jump chain is given in Figure \ref{fig:MetastabilityGraph}.

\begin{figure}[h]
	\centering
	\includegraphics[width=.95\textwidth]{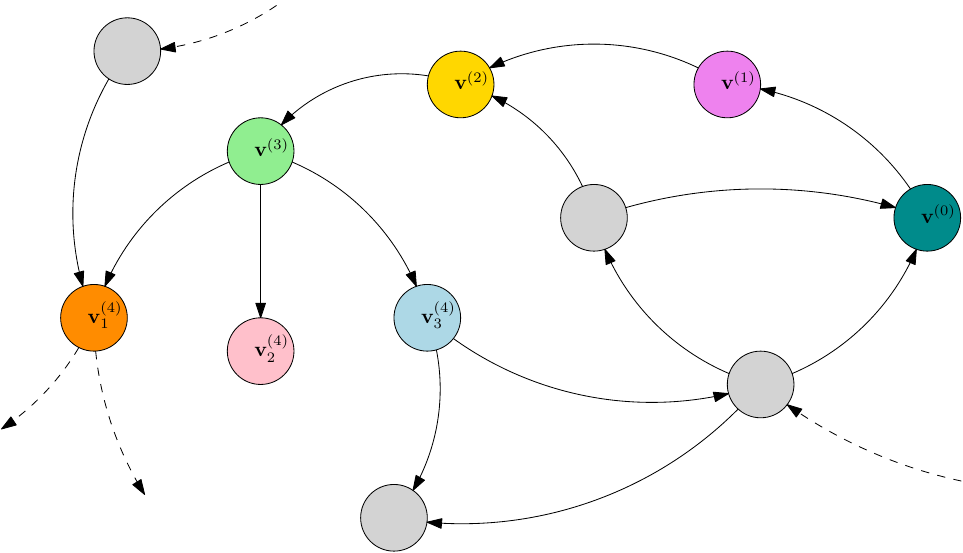}
	\caption{Metastability graph $\GG_\ESC$ including a jump chain $(\v_k)_{k\geq 0}$, where $\v^{(4)}_i=\v_\ESC (\v^{(3)},w_i)$, for $V_\mut(\v^{(3)})=\{w_1,w_2,w_3\}$.}
	\label{fig:MetastabilityGraph}
\end{figure}
	
After this general description of the multi-scale jump chain we can now easily elaborate the true Markov jump process on each time scale. To be more precise, for each stability degree $L>\alpha$, we are looking for the limit process of $N^K_w(t/K\mu_K^L)/K$, for $K\to\infty$. The support of this process jumps between sets of coexisting traits of sufficiently high stability degree, which can only be exited on this time scale. In this context, we define the level sets of equal stability degree as
\begin{align}
\SS^L:=\dset{\v\subseteq V :  LVE_+(\v)=\{\bar{n}(\v)\}, L(\v)=L}.
\end{align}

Note that, for $L>\alpha$, a stability degree of $L(\v)=L$ ensures that the coexisting traits $\v$ allow for an asymptotic ESC, see Remark \ref{rem:ESCstability}.

As the state space for the limiting jump process, we introduce the \emph{$L$-scale graph} $\GG^L$, which is a collapsed version of $\GG_\ESC$. The vertex set consists of all ESCs that are stable enough to be visible on the respective time scale. Therefore, we set
\begin{align}
	\label{eq:Def_VL}
	\VV^L:=\bigcup_{L'\geq L}\mathcal{S}^{L'}.
\end{align}

Note that it is possible that the process jumps into an ESC  $\v\in \mathcal{S}^{L'}$, for $L'>L$, on the $1/K\mu_K^L$-time scale. However, there is no possibility to escape from those on this time scale, which means that these are absorbing states.

Edges $\EE^L$ in $\GG^L$ represent possible transitions of the limiting process. To construct these, we study the limiting jump chain from Corollary \ref{Cor:kProcess}.

In order to use the corollary and in particular the process $(\v^{(k)})_{k\geq0}$, we have to ensure that, for fixed $L>\alpha$, this process always reaches an ESC of stability degree at least $L$ in finitely many steps. 

\begin{assumption}
	\label{Ass:NoCycles}
	\begin{align}
	\forall\v\in\mathcal{S}^L: \Prob{\exists k\in\N_{>0}:L(\v^{(k)})\geq L\vert\v^{(0)}=\v}=1
	\end{align}
\end{assumption}

Note that, if this assumption is satisfied for some fixed $L$, this has no implications for the validity for different $L'\neq L$. This is due to the fact, that only the initial conditions $\v\in\SS^L$ are considered. One can easily think of counterexamples where $\GG_\ESC$ is non-connected such that there may be cycles of lower time scale but there is no danger to run into them. For a broader discussion of the assumption we refer to the Examples \ref{ex:6_Assumption} and \ref{ex:7_Assumption}.

\begin{remark}
	\label{rem:EscapeFromCycle}
	If the process runs into a cycle or stable cluster on a lower time scale, there are still possibilities to escape from these by accelerating and looking at higher time scales. The detailed description of such behaviour is much more involved. This is mainly due to technical reasons: Errors accumulate in the approximation of each transition step. As long as it is ensured that the system reaches a (sufficiently stable) ESC after finitely many steps, these errors can be iteratively bounded to ensure convergence. This however fails if the number of lower time scale transitions between higher time scale jumps is not bounded. Heuristically, if one can observe ergodic behaviour on the $L'$-scale graph, for some $L'<L$, transitions out of the ergodic cluster will occur along one of the shortest fitness valleys of width $L$. Transition rates will be weighed according to the stationary distribution on states in $\mathcal{S}^{L'}$ and the transition takes a time of order $1/K\mu_K^L$. Rather than defining vertices of $\mathcal{G}^L$ as single sets of coexisting traits in $\mathcal{S}^L$, one would then choose communication classes of such sets in $\mathcal{S}^{L'}$ (for possibly multiple $L'<L$) that support an ergodic stationary distribution. Rigorously justifying this argument is a topic of current and future research of the authors.
\end{remark}

Now asking for possible jumps in $\mathcal{G}^L$ we have to respect again the principle that jumps on lower time scales are absorbed in those happening on the $1/K\mu_K^L$-time scale. This means that the critical event for a transition starting in $\v\in\mathcal{S}^L$ is to escape from $\v$, which needs a time of order $1/K\mu_K^L$. Compared to that, the subsequent transitions in $\mathcal{G}_\ESC$ until reaching again a state $\w$ of stability at least $L(\w)\geq L$ take place in very short time. Therefore we say that the (directed) eged $(\v,\w)$ is in $\mathcal{E}^L$ if and only if $L(\v)=L$ and there exists a finite path $\Gamma:\v\to\w$ in $\mathcal{G}_\ESC$ such that $L(\Gamma_i)<L,\ \forall 1\leq i < \abs{\Gamma}$.

The probability of possible transitions $(\v,\w)\in\mathcal{E}^L$ is then the sum over all possible paths $\Gamma$ that give rise to this edge, while the probability of taking a particular path is easily computed as the product of its segments in $\mathcal{G}_\ESC$.
\begin{align}
\label{eq:ProbLTimescale}
p^L(\v,\w):=
\sum_{\substack{\Gamma:\v\to\w\\ L(\Gamma_i)<L,\ \forall 1\leq i < \abs{\Gamma}}}
\prod_{i=1}^{\abs{\Gamma}}p(\Gamma_{i-1},\Gamma_{i})
\end{align}
For an explanatory computation of these probabilities we refer to the Examples \ref{ex:8_LargeSmallLarge} and \ref{ex:9_MultEscLowStab}.

The transition rate for the jumps on the $1/K\mu^L$ time scale are then given by the over-all rate to escape from $\v$ weighted with the transition probability to end in $\w$.
\begin{align}
\label{eq:RatesLTimescale}
\mathcal{R}^L(\v,\w):=R(\v)p^L(\v,\w)
\end{align}

Now we are prepared to formulate the main result, i.e.\ the convergence to a Markov jump process on different time scales.

\begin{theorem}
	\label{Thm:Lscale}
	Let $L>\alpha$ such that Assumption \ref{Ass:NoCycles} holds true and take $\v^L(0)\in\VV^L$.
	Suppose the same assumptions as in Theorem \ref{Thm:main} are satisfied for $\v=\v^L(0)$ and assume that the  $\ln K$-dynamics behave as in Corollary \ref{Cor:main}, for every $\v\in\bigcup_{L'\leq L}\SS^{L'}$ and all $w\in V_\mut(\v)$.
	Then, for all $T<\infty$, the rescaled process $(N^K_v(t/K\mu_K^L)/K,\ v\in V,\ t\in[0,T])$ converges in the sense of finite dimensional distributions to a jump process of the form
	\begin{align}
	\mathcal{N}_v^L(t)=\ifct{v\in \v^L(t)}\bar{n}_v(\v^L(t)), && v\in V,\ t\in[0,T].
	\end{align}
	Here $(\v^L(t),\ t\in[0,T])$ is a Markov jump process on the $L$-scale graph \mbox{$\mathcal{G}^L=(\mathcal{V}^L,\mathcal{E}^L)$}, with transition rates given by \eqref{eq:RatesLTimescale}.
\end{theorem}

\begin{remark}
	We like to point out that Assumption \ref{Ass:NoCycles} does not exclude the cases where we have cycles in $\GG^L$, i.e.\ on the time scale $1/K\mu_K^L$. It only prevents the process from running into a cycle of lower time scale. We even allow for self connecting edges, i.e.\ edges of the form $(\v,\v)$.
\end{remark}

%%%%%%%%%%%%%%%%%%%%%%%%%%%%%%%%%%%%%%%%%%%%%%%%%%%%%%%%%%%%%%%%%%%%%

\section{Interesting examples}
\label{sec:3_examples}

In this chapter, we present and analyse a variety of examples that aim to highlight different aspects of the complicated dynamics covered in our main results. The first two
%four 
examples are dedicated to single transition steps from one ESC to another, applying the results of Theorem \ref{Thm:main} and Corollary \ref{Cor:main}. The next three examples focus on the metastability graph $\GG_\ESC$ that is constructed in Corollary \ref{Cor:kProcess} and we study two cases that are concerned with Assumption \ref{Ass:NoCycles}. The final two examples are focussed on applications of Theorem \ref{Thm:Lscale}, studying the limiting Markov jump processes on different time scales as well as the $L$-scale-graphs $\GG^L$.

In order to give a manageable and clear description of the dynamically changing fitness landscape, we introduce some new notation that helps to simplify the set up of the examples.

\begin{definition}
	We speak of a \emph{regime of equal competition} if and only if\linebreak $c(v,w)\equiv \text{const}>0$, for all $v,w\in V$.
\end{definition}

This is by no means a necessary assumption to produce the studied phenomena, however, it allows us to characterise the fitness landscape in a much simpler way. In the case of equal competition, the invasion fitness of a trait $w$ with respect to a single resident trait $v$ is fully characterised by
\begin{align}
	f(w,v)=r(w)-r(v),
\end{align}
where we set $r(v):=b(v)-d(v)$ as the \textit{individual fitness} of trait $v$, i.e.\ its net growth rate in the absence of competitive interactions. As a consequence, traits $w$ with higher individual fitness than the resident $v$ are able to invade the population. Hence, instead of specifying the invasion fitnesses for all possible resident traits, the fitness landscape is fully described by the individual fitnesses $r(v)$.

To specify the fitness relations between different traits - in particular in the case of non-equal competition - we introduce the following notation.
\begin{definition}
	For $v,w\in V$, we write $v\ll w$ if and only if $f(w,v)>0$ and $f(v,w)<0$.
	Moreover, we write $v_1,...,v_k\ll w_1,...,w_l$ whenever $v_i\ll w_j$, for all $1\leq i\leq k$ and $1\leq j\leq l$.
\end{definition}
This reflects the case where the equilibrium of the Lotka-Volterra system involving $v$ and $w$ is the monomorphic equilibrium $\bar{n}(w)$ of $w$. In other words $w$ can invade the $v$ population and fixate.

%%%%%%%%%%%%%%%%%%%%%%%%%%%%%%%%%%%%%%%%%%%%
\subsection{Single transition steps}

\subsubsection{A first example with multiple mutation paths}
	\begin{example}
		\label{ex:0_combined}
		Let us consider the directed graph $G$ depicted in Figure \ref{fig:example0_setting}. Assume equal competition and the individual fitness $r$ plotted in Figure \ref{fig:example0_setting}. Moreover, let $\a\in(1,2)$.
	\end{example}
	\begin{figure}[h]
		\centering
		\includegraphics[scale=0.6]{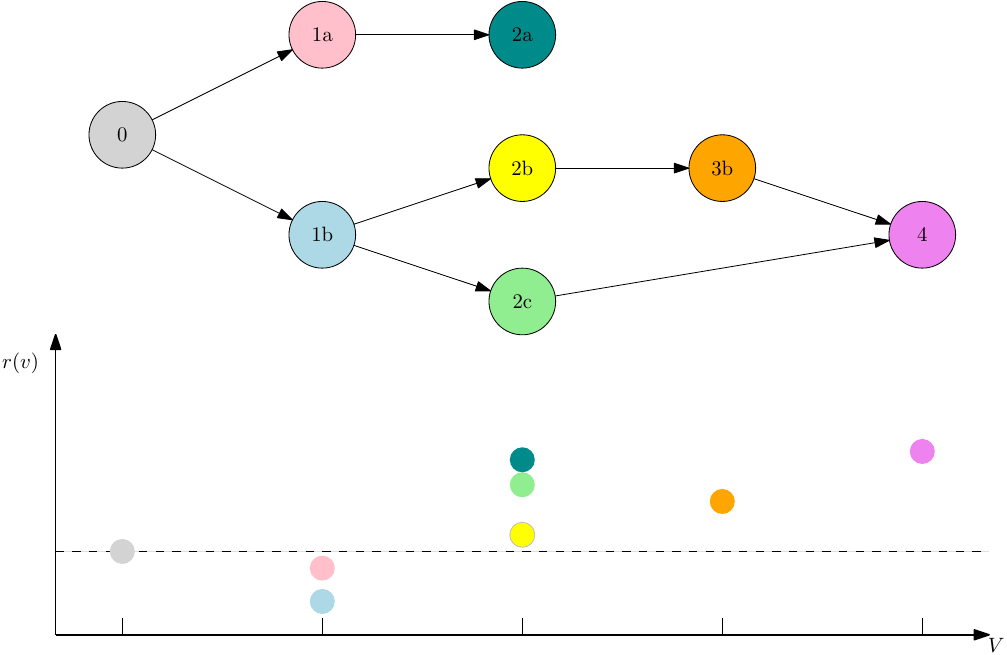}
		\caption{Trait graph $G$ and fitness landscape in terms of individual fitness $r$ of Example \ref{ex:0_combined}.}
		\label{fig:example0_setting}
	\end{figure}
		In this case, the initial resident trait 0 has stability degree $L(\{0\})=2>\a$. This is due to the fact that traits $1a$ and $1b$ are unfit in presence of the resident, while traits $2a$, $2b$ and $2c$ are fit, with connecting paths $\g^{A}=(0,1a,2a)$, $\g^{B}=(0,1b,2b)$ and $\g^{C}=(0,1b,2c)$ of length 2 respectively.
		Therefore, we have the possible mutant candidates $V_\mut(\{0\})=\dset{2a,2b,2c}$.
		An application of Theorem \ref{Thm:main} yields that we can observe a new fixating trait at rescaled time $T^K_\fix K\mu^2_K$, which is distributed approximately as a exponential random variable with rate $R(\{0\})=R(\{0\},2a)+R(\{0\},2b)+R(\{0\},2c)$. The probability for say trait $2b$ to be the trait that fixates in the population and triggers the stopping time is $R(\{0\},2b)/R(\{0\})$.
		
		Asking for the new ESCs, which are reached after fixation, we have to take into account the subsequent evolution on the $\ln K$ time scale. This allows for jumps towards traits of higher fitness, which are in the mutation spreading neighbourhood, i.e.\ direct neighbours in this case. Therefore, we end up with
		\begin{align}
			\v_\ESC (\{0\},2a)=\{2a\}\quad
			\v_\ESC (\{0\},2b)=\{4\},\quad
			\v_\ESC (\{0\},2c)=\{4\}.
		\end{align}
		In particular, note that $\{2b\},\{2c\}$ are not ESCs and thus not part of the metastability graph $\GG_\ESC$ as plotted in Figure \ref{fig:example0_graphs}.
		
		This puts us into the setting where the sum in \eqref{eq:TransitionRate} becomes relevant. In particular, despite the micro-evolutionary branching from $1$ into $2b$ and $2c$ in the trait graph $G$ , there is no such branching on the macro-evolutionary level in $\GG_\ESC$. There, we only observe a transition from $\{0\}$ to $\{4\}$.
		Note also that the different path lengths of $2a\to 4$ and $2b\to 4$ do not matter for the asymptotics of the time $T_\ESC$ until stabilising in the new ESC. This is because this time is dominated by the waiting time $T_\fix$ for the first fixation of a fit mutant trait. Since $L(\{0\})=2$, this time is of order $1/K\mu_K^2$ and thus absorbs the much faster $\ln K$ evolution.
		
		Note that, since all transitions between ESCs occur on the time scale $1/K\mu_K^2$ here, the metastability graph $\GG_\ESC$ agrees with the 2-scale graph $\GG^2$.
		\begin{figure}[h]
			\centering
			\includegraphics[scale=0.55]{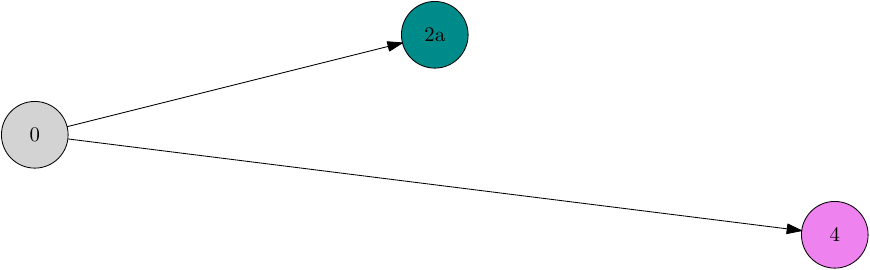}
			\caption{Metastability graph $\GG_\ESC$ and 2-scale graph $\GG^2$ of Example \ref{ex:0_combined}}
			\label{fig:example0_graphs}
		\end{figure}
	
\subsubsection{An ESC with coexistence}
	Since in this paper we discuss the occurrence of metastable behaviour in a rather general setting, we like to point out that Definition \ref{def:ESC} explicitly allows for ESCs $\v$ that consist of several coexisting traits. This clearly enlarges the mutation spreading neighbourhood $V_\a(\v)$ and changes the set of mutant candidates $V_\mut$ in a non-trivial way.
%%%% EX 4 %%%%		
	\begin{example}
		\label{ex:4_Coex}
		Let us consider the undirected graph $G$ depicted in Figure \ref{fig:example4}. Let $\a\in(1,2)$ and consider a fitness landscape that satisfies
		\begin{align}
		f(0,3),f(3,0)>0,\label{Ass_coex}\\
		f(1,\{0,3\}),f(2,\{0,3\})<0,\\
		f(4,\{0,3\}),f(5,\{0,3\})>0,\\
		0,1,2,3\ll 4,5\\
		1,2\ll0,3\\
		f(4,5),f(5,4)<0,
		\end{align}
		and allows for no polymorphic coexistence equilibria apart from $\{0,3\}$.
		Moreover, assume that the unique stable equilibrium of the Lotka-Volterra system involving traits $\{0,3,4\}$ is $\bar{n}(4)$  and the same is true for 5 replacing 4.
	\end{example}
	\begin{figure}[h]
		\centering
		\includegraphics[scale=0.6]{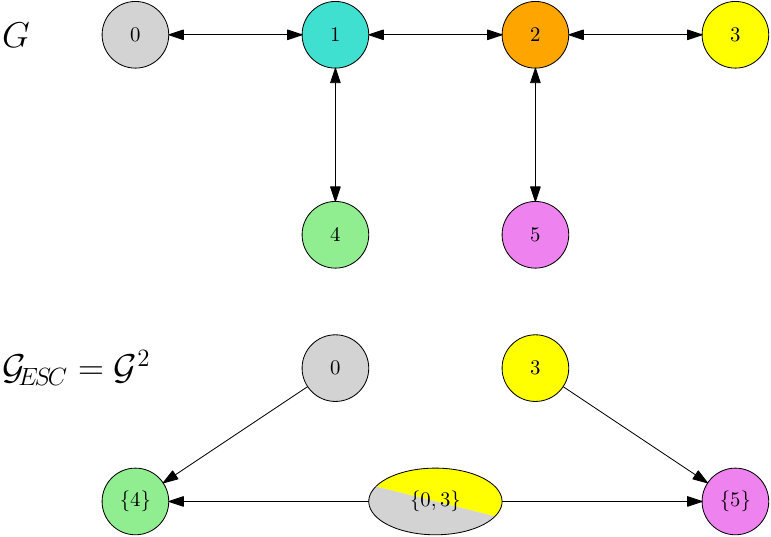}
		\caption{Trait graph $G$ and metastability graph $\GG_\ESC$ (which agrees with the 2-scale graph $\GG^2$) of Example \ref{ex:4_Coex}}
		\label{fig:example4}
	\end{figure}

	Checking for traits that do not have any fitter direct neighbours, and hence do not allow for transitions on the $\ln K$-time scale, the monomorphic ESCs in this case correspond to $\{0\}$, $\{3\}$, $\{4\}$, and $\{5\}$. Classical results on Lotka-Volterra systems yield that under assumption \eqref{Ass_coex} traits 0 and 3 can coexist, i.e.\ $\bar{n}(\{0,3\})\in\R_{>0}^2$. Now the mutation spreading neighbourhood is given by $V_\a(\{0,3\})=\{0,1,2,3\}$. Apart from the resident traits themselves, those traits are by assumption unfit with respect to $\{0,3\}$ and thus $\{0,3\}$ allows for an ESC.
	
	 Looking for the stability degree and possible mutant candidates, the assumptions on the fitness landscape imply that
	\begin{align}
		L(\{0,3\})=2\quad\text{and}\quad V_\mut(\{0,3\})=\{4,5\}.
	\end{align}
	By Theorem \ref{Thm:main}, we can observe a fixating mutant population of one of the traits\linebreak $w\in V_\mut(\{0,3\})$ on the time scale $1/K\mu_K^2$.  The corresponding rates are given by
	\begin{align}
		R(\{0,3\},4)=\bar{n}_0(\{0,3\})\frac{b(0)m(0,1)}{\abs{f(1,\{0,3\})}}b(1)m(1,4)\frac{f(4,\{0,3\})}{b(4)}, &&\text{for } w=4,\\
		R(\{0,3\},5)=\bar{n}_3(\{0,3\})\frac{b(3)m(3,2)}{\abs{f(2,\{0,3\})}}b(2)m(2,5)\frac{f(5,\{0,3\})}{b(5)}, &&\text{for } w=5.
	\end{align}
	Note that, although there are also paths connecting $3\to4$ and $0\to5$, only the paths of shortest length $\abs\g=2$ do have an impact on the above rates.
	
	To conclude this example, we see that both mutant traits $4$ and $5$ are fit enough to invade the coexisting resident population.  Overall, we obtain the metastability graph $\GG_\ESC$ pictured in Figure \ref{fig:example4}, which in this case agrees with the 2-scale graph $\GG^2$. Note that the traits 0 and 3 appear both as monomorphic ESCs, as well as a polymorphic coexistence ESC.

\subsection{Successive metastable transitions}

\subsubsection{Self connection in $\GG_\ESC$}
	By definition of an ESC, the first fixating mutant has a distance of at least $\lalpha+1$ from the corresponding resident traits. Despite this fact, the $\ln K$-mechanism triggered by such a mutant may lead to a new ESC that is closer to the old one than $\lalpha+1$. It can even be the same and thus give rise to a self-connecting edge in $\GG_\ESC$
%%%% EX 5 %%%%		
	\begin{example}
		\label{ex:5_Selfconn}
		Let us consider the directed graph depicted in Figure \ref{fig:example5} and take $\a\in(1,2)$. Consider a fitness landscape that satisfies
		\begin{align}
		0\ll 2\ll 4\ll 5\ll 2,\\
		1\ll2,\qquad3\ll4,\\
		f(1,0),f(3,2),f(3,5)<0
		\end{align}
		and assume that there are no polymorphic coexistence equilibria.
	\end{example}
	\begin{figure}[h]
		\centering
		\includegraphics[scale=0.6]{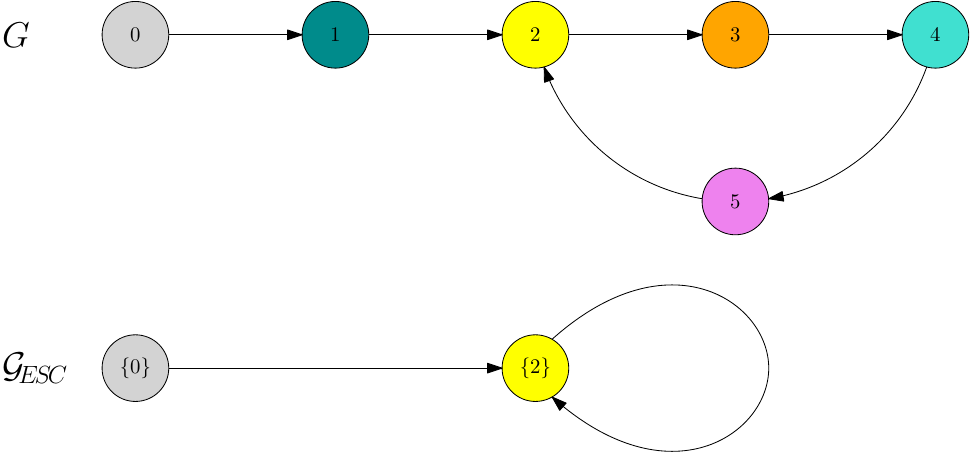}
		\caption{Trait graph $G$ and metastability graph $\GG_\ESC$ of Example \ref{ex:5_Selfconn}}
		\label{fig:example5}
	\end{figure}

	After a first jump from $\v^{(0)}=\{0\}$ to $\v^{(1)}=\{2\}$ on the time scale $1/K\mu_K^2$, the next fixating mutant is of trait $4$ and arises on the same time scale. The chosen fitness landscape ensures that it grows and can invade the population of trait $2$ within a $\ln K$-time. Since $\a\in(1,2)$, we obtain a non-vanishing population of trait $5$ on the same time scale, which can grow as soon as trait 4 is the new resident trait. Due to its positive invasion fitness, $5$ invades the trait $4$ population. Finally, the same argument applies for an invasion by trait $2$, where we then get stuck in because $\{2\}$ is an ESC of stability degree $L(\{2\})=2>\a$.
	
	Overall, we obtain that
	\begin{align}
		 \v^{(2)}=\v_\ESC (\{2\},4)=\{2\}.
	\end{align}
	In view of Definition \ref{def:MetastabilityGraph}, this gives rise to the self-connecting edge $(\{2\},\{2\})\in\GG_\ESC$, which is illustrated in Figure \ref{fig:example5}.

\subsubsection{On Assumption \ref{Ass:NoCycles}}
		Since the assumption that prevents 
		the process from getting stuck on a slower time scale is somewhat involved, we give two examples. First, we illustrate in Example \ref{ex:6_Assumption} that Assumption \ref{Ass:NoCycles} may hold true even if there is a cycle in the metastability graph $\GG_\ESC$. Second, we slightly modify the trait graph $G$ and the fitness landscape to get Example \ref{ex:7_Assumption}, where Assumption \ref{Ass:NoCycles} is not satisfied, and explain why this leads to difficulties.
%%%% EX 6 %%%%		
	\begin{example}
		\label{ex:6_Assumption}
		Let us consider the directed graph depicted in Figure \ref{fig:example6}. Let $\a\in(0,1)$ and consider a fitness landscape that satisfies
		\begin{align}
		0\ll 2\ll 3\ll 4\ll 6,\\
		1\ll2,\qquad  5\ll6,\\
		3\ll 7\ll 2,\\
		f(1,0),f(5,4)<0
		\end{align}
		and assume that there are no polymorphic coexistence equilibria.
	\end{example}
	\begin{figure}[h]
		\centering
		\includegraphics[scale=0.6]{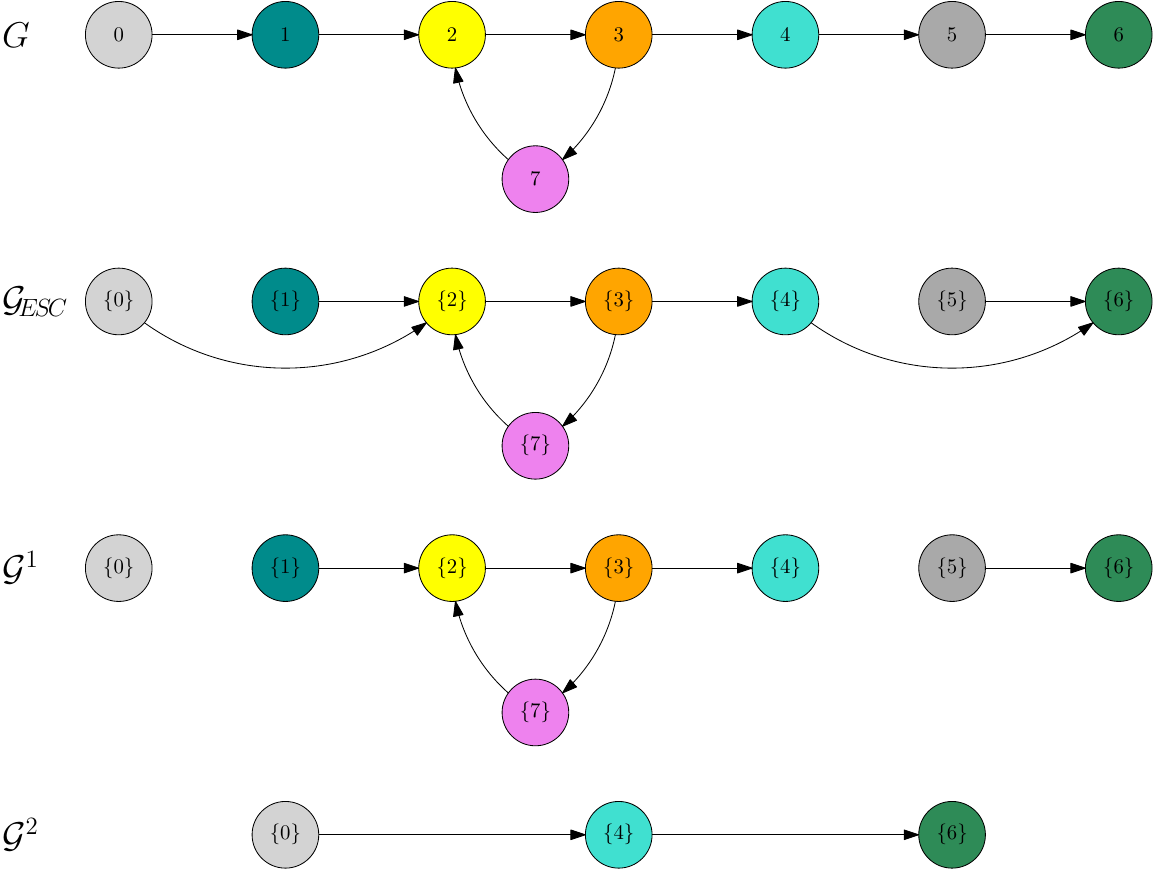}
		\caption{Trait graph $G$, metastability graph $\GG_\ESC$ and $L$-scale graphs $\GG^1$ and $\GG^2$ of Example \ref{ex:6_Assumption}}
		\label{fig:example6}
	\end{figure}
	
	Let us first remark that, because of $\a\in(0,1)$, we are in the regime of the trait substitution sequence (cf.\ \cite{Cha06}). This means that we can neglect the $\ln K$-algorithm. In particular, if $v\ll w$, for some $w\in V_\mut(\{v\})$, then $\v_\ESC (\{v\},w)=\{w\}$.
	
	With this knowledge, let us construct the jump chain step by step. The first two jumps are determined easily, noting that
	\begin{align}
		\v^{(0)}&=\{0\}, & L(\{0\})&=2, & V_\mut(\{0\})&=\{2\},\\
		\v^{(1)}&=\{2\}, & L(\{2\})&=1, & V_\mut(\{2\})&=\{3\},\\
		\v^{(2)}&=\{3\}, & L(\{3\})&=1, & V_\mut(\{3\})&=\{4,7\}.
	\end{align}
	For the third jump, there are two possible triggering mutants. If trait $7$ fixates first, the process jumps to the ESC $\v^{(3)}=\{7\}$ and then returns to $\v^{(4)}=\{2\}$, all on the time scale $1/K\mu_K$.
	If instead trait $4$ fixates earlier, the jump chain continues to $\v^{(3)}=\{4\}$ within a time of order $1/K\mu_K$ and then to $\v^{(4)}=\{6\}$ on the time scale $1/K\mu_K^2$, since $f(5,4)<0$.
	
	Mentioning that $V_\mut(\{1\})=\{2\}$ and $V_\mut(\{5\})=\{6\}$ gives us the metastability graph drawn in Figure \ref{fig:example6}.
	
	To check whether Assumption \ref{Ass:NoCycles} is satisfied, we decompose the set of ESCs $\VV_\ESC$ according to the stability degree,
	\begin{align}
		&\SS^1=\dset{\{1\},\{2\},\{3\},\{5\},\{7\}},&&\SS^2=\dset{\{0\},\{4\}},&&\SS^\infty=\dset{\{6\}}.
	\end{align}
	
	For all $\v\in\SS^1$, one directly sees that an ESC of the same or a higher stability is reached after one jump with probability one. Thus the assumption is true for $L=1$ and we can construct the graph $\GG^1$ as drawn.
	
	In the case of $L=2$, for $\v^{(0)}=\{4\}$, we obtain that with probability one the process jumps to $\v^{(1)}=\{6\}$, which is of higher stability. Finally, we have to check the most involved case of $\v^{(0)}=\{0\}$. From the metastability graph we identify $\v=\{4\}$ as the only reachable ESC of degree $L\geq2$. Due to the branching at $\{3\}$, we have to ensure that the process does not get stuck in a cycle of $(\{2\},\{3\},\{7\},\{2\})$ for infinitely many steps. We can see that
	\begin{align}
		&\Prob{\forall\ k\in\N_{>0}: \v^{(k)}\neq\{4\} \vert \v^{(0)}=\{0\}}=0
	\end{align}
	since the number of cycles that run through before exiting towards $\{4\}$ has a geometric law with success probability $p(\{3\},\{4\})>0$.
	Therefore,  Assumption \ref{Ass:NoCycles} also holds true for $L=2$. This yields the $L$-scale graph $\GG^2$, depicted in Figure \ref{fig:example6}.

	Let us now modify the example by inserting an additional trait $8$, that can be viewed as an intermediate unfit mutation between $3$ and $4$. Moreover, for the sake of clarity, we cut off the traits $5$ and $6$.
%%% EX 7

	\begin{example}
		\label{ex:7_Assumption}
		Let us consider the directed graph depicted in Figure \ref{fig:example7} and let $\a\in(0,1)$.
		Consider a fitness landscape that satisfies
		\begin{align}
		0\ll 2\ll 3\ll 4,\\
		3\ll 7\ll 2,\\
		1\ll2,\qquad8\ll4,\\
		f(1,0),f(8,3)<0
		\end{align}
		and assume that there are no polymorphic coexistence equilibria.
	\end{example}
	\begin{figure}[h]
		\centering
		\includegraphics[scale=0.6]{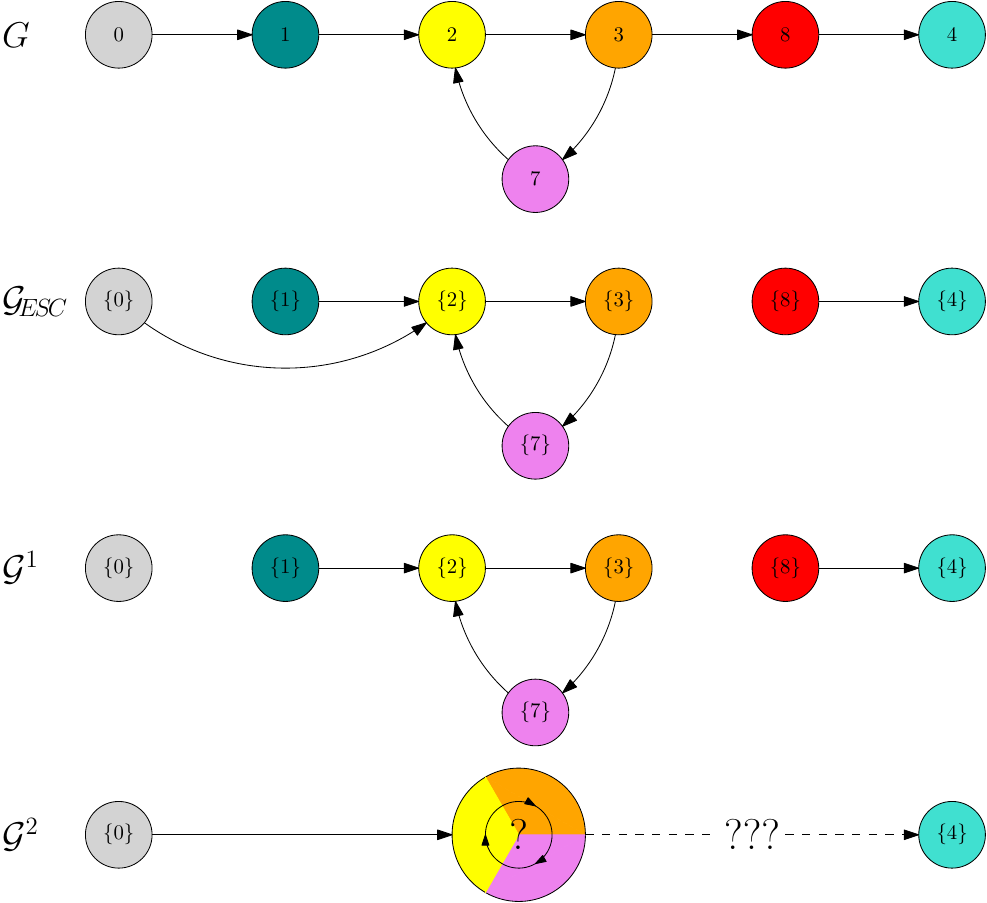}
		\caption{Trait graph $G$, metastability graph $\GG_\ESC$ and $L$-scale graphs $\GG^1$ and $\GG^2$ of Example \ref{ex:7_Assumption}}
		\label{fig:example7}
	\end{figure}

	Since we only changed the trait graph $G$ slightly, also the metastability graph $\GG_\ESC$ stays almost the same. Apart from the omitted traits $5$ and $6$, the main difference is that the valley from the ESC $\{3\}$ to the fit mutant $4$ is now of width 2. Therefore, trait $4$ is no longer one of the nearest fit traits to trait $3$ and the set of possible mutants gets reduced to	$V_\mut(\{3\})=\{7\}$. In particular, there is no longer an edge $(\{3\},\{4\})$ in the metastability graph.
	
	To check whether Assumption \ref{Ass:NoCycles} is satisfied, we again separate the stability classes
	\begin{align}
		\SS^1=\dset{\{1\},\{2\},\{3\},\{7\},\{8\}},&&\SS^2=\dset{\{0\}},&&\SS^\infty\dset{\{4\}}.
	\end{align}
	
	For $L=1$, it is again easy to see from $\GG_\ESC$ that the assumption holds true. To check this for $L=2$, we have to consider how the process can get from the initial ESC $\{0\}$ to some ESC of at least the same stability degree. This is not possible since the only candidate would bee $\{4\}$, which is not reachable since the metastability graph is disconnected. As a conclusion, Assumption \ref{Ass:NoCycles} is not satisfied for $L=2$ and thus we can neither construct the $L$-scale graph $\GG^2$ nor apply Theorem \ref{Thm:Lscale}.
	
	\begin{remark}
		Although the population process gets stuck in a cycle between of the ESCs $\{2\},\{3\},\{7\}$ for infinite time, we expect that it might escape through the fitness valley $3\to8\to4$ eventually, when looking at the time scale $1/K\mu_K^2$. This is due to the fact that, from the microscopic point of view, it is possible to observe mutants of trait $4$ in the phases where $3$ is the resident trait. Indeed, those mutants appear with a much smaller rate than those of trait $7$, but since these phases occure infinitely often, it should only be a question of acceleration to escape from this cycle (c.f.\ Remark \ref{rem:EscapeFromCycle}).
	\end{remark}

\subsubsection{Collapse on higher time scales}
	In the two final examples, we demonstrate how paths in the metastability graph that pass through ESCs of different stability degree collapse to a single edge in the $L$-scale graph when focussing on a particular time scale. To this end we start with an example of a simple linear trait graph with multiple successive fitness valleys of different length. The second example allows for a branching in the metastability graph, which again vanishes in the $L$-scale graph.
	
%%%% EX 8 %%%%	
	\begin{example}
		\label{ex:8_LargeSmallLarge}
		Let us consider the directed graph $G$ depicted in Figure \ref{fig:example8_setting}. Assume equal competition and the individual fitness $r$ plotted in Figure \ref{fig:example8_setting}. Moreover, let $\a\in(1,2)$.
	\end{example}
	\begin{figure}[h]
		\centering
		\includegraphics[scale=0.55]{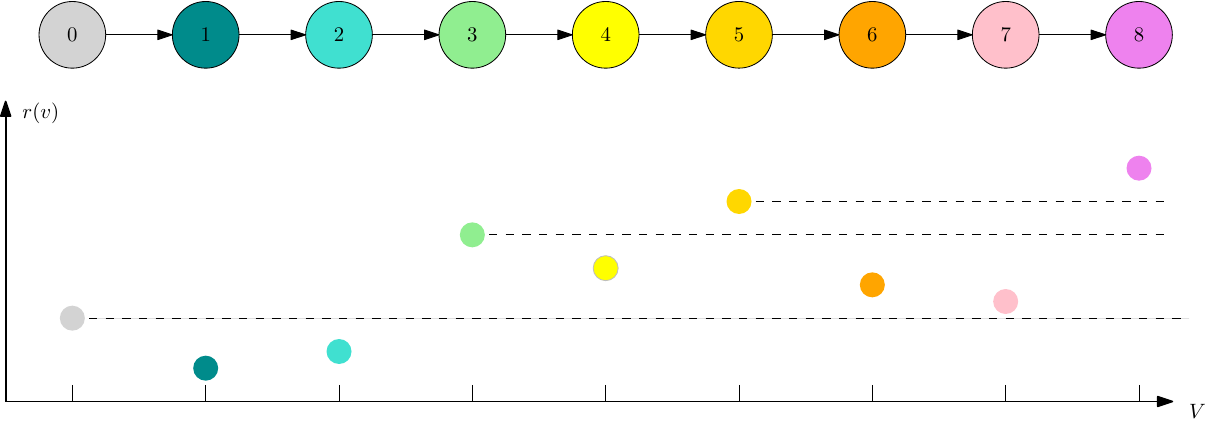}
		\caption{Trait graph $G$ and fitness landscape in terms of individual fitness $r$ of Example \ref{ex:8_LargeSmallLarge}}
		\label{fig:example8_setting}
	\end{figure}

	Due to the linear and directed structure of the trait graph, we can extract the fitness valleys and thus the stability degrees directly from the plotted individual fitness $r$. The jump chain $(\v^{(k)})_{k\geq 0}$ is the deterministic sequence
	\begin{align}
		\v^{(0)}=\{0\}, &&\v^{(1)}=\{3\}, &&\v^{(2)}=\{5\}, &&\v^{(3)}=\{8\}.
	\end{align}
	This is reflected in the metastability graph drawn in Figure \ref{fig:example8_graphs}. Note that $\{6\}$ is also an ESC of stability degree $2$, but it cannot be reached starting from $\{0\}$.
	
	Let us now have a look at the $L$-scale-graphs, i.e.\ at how the limiting jump process evolves when fixing a particular time scale. To this end we focus on the sets of ESCs of equal stability degree, namely
	\begin{align}
		\SS^2=\dset{\{3\},\{6\}}, && \SS^3=\dset{\{0\},\{5\}}, && \SS^\infty=\dset{\{8\}}.
	\end{align}
	Following our construction in \eqref{eq:Def_VL}, the $L$-scale-graph $\GG^2$ consists of the vertices\linebreak $\VV^2=\dset{\{0\},\{3\},\{5\},\{6\},\{8\}}$. Since all but $\{3\}$ and $\{6\}$ are of stability degree higher than $L=2$, the only edges are $\EE^2=\dset{(\{3\},\{5\}),(\{6\},\{8\})}$.
	
	The construction of the edges of $\GG^3$ is far more interesting. In particular, starting in the initial ESC $\v^{(0)}=\{0\}$, we cannot simply take the edge $(\{0\},\{3\})$ from the metastability graph since $L(\{3\})<3$ and thus $\{3\}$ is not stable enough. Instead, we have to consider the whole path $\G=(\{0\},\{3\},\{5\})$ until an ESC of higher stability is reached. This is because the second jump of $\G$ happens much faster (more precisely on the time scale $1/K\mu_K^2$) and hence becomes absorbed in the slower first jump when rescaling the process with $1/K\mu_K^3$. This gives us one edge in $\EE^3$. The second one is given by the jump $(\{5\},\{8\})$. Since $L(\{8\})=\infty$, no further evolution is possible here.
	
	Overall, these considerations lead to the pictures of $\GG^2$ and $\GG^3$ in Figure \ref{fig:example8_graphs}.
	\begin{figure}[h]
		\centering
		\includegraphics[scale=0.6]{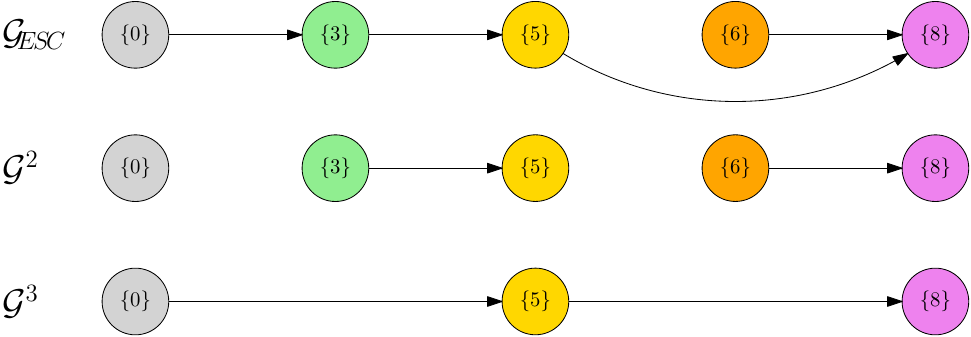}
		\caption{Metastability graph $\GG_\ESC$ and $L$-scale graphs $\GG^2$ and $\GG^3$ of Example \ref{ex:8_LargeSmallLarge}}
		\label{fig:example8_graphs}
	\end{figure}

%%%% EX 9 %%%%	
\begin{example}
	\label{ex:9_MultEscLowStab}
	Let us consider the directed graph $G$ depicted in Figure \ref{fig:example9_setting}. Assume equal competition and the individual fitness $r$ plotted in Figure \ref{fig:example9_setting}. Moreover, let $\a\in(1,2)$.
\end{example}
\begin{figure}[h]
	\centering
	\includegraphics[scale=0.58]{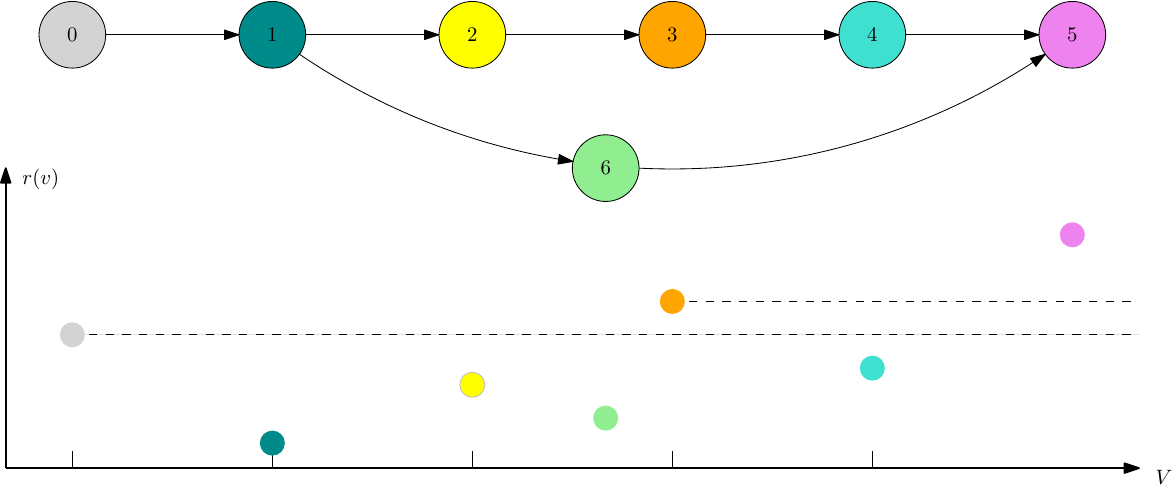}
	\caption{Trait graph $G$ and fitness landscape in terms of individual fitness $r$ of Example \ref{ex:9_MultEscLowStab}}
	\label{fig:example9_setting}
\end{figure}

Starting with the resident population in $\v^{(0)}=\{0\}$, we can directly extract from the plotted individual fitness $r$ that only the traits $3$ and $5$ have positive invasion fitness. Moreover, both can be reached via a path of length $\abs{\g}=3$, namely
\begin{align}
\g^A=(0,1,2,3), && \g^B=(0,1,6,5).
\end{align}
Hence, we associate to this ESC the stability degree $L(\{0\})=3$ and the set of mutant candidates $V_\mut(\{0\})=\dset{3,5}$.

If trait $5$ fixates first, there is no further evolution and we end with $\v_\ESC (\{0\},5)=\{5\}$. In the case where trait $3$ fixates, it can grow and becomes macroscopic. Moreover, since $\a\in(1,2)$, the population of trait $4$ grows by frequent incoming mutants. However, due to its negative invasion fitness with respect to the resident $\{0\}$ and later against the macroscopic population $\{3\}$, it cannot invade. Hence $\v_\ESC (\{0\},3)=\{3\}$ is the corresponding ESC and is of stability degree $L(\{3\})=2$. From thereon, only trait $5$ is a fit reachable mutant, which arises after a waiting time of order $\OO(1/K\mu_K^2)$ and replaces $3$ as an ESC. Those three jumps form the edges of the drawn metastability graph $\GG_\ESC$ in Figure \ref{fig:example7_graphs}.

	The $L$-scale-graph $\GG^2$ is constructed easily whereas the really interesting behaviour occurs when asking for the $\GG^3$. Since $L(\{0\})=3$, the jumps $(\{0\},\{3\})$ and $(\{0\},\{5\})$ happen on the visible time scale. The latter one is clearly also an edge in $\GG^3$, due to the high stability of the final ESC $L(\{5\})=\infty$. However, in case of the former, the ESC that the process jumps to is of smaller stability, i.e.\ $L(\{3\})=2$. Therefore, the next jump $(\{3\},\{5\})$ directly occurs within a time that vanishes under rescaling. The path $\G=(\{0\},\{3\},\{5\})$ in $\GG_\ESC$ thus yields an edge $(\{0\},\{5\})$ for $\GG^3$. This edge already exists and we do not allow for double edges in $\GG^L$. However, the two edges are merged in the sense of adding up the transition rates and probabilities as in \eqref{eq:ProbLTimescale}.
	
	Overall, we see that even a branching in the metastability graph can disappear when multiple paths collapse to the same edge on a particular time scale.

\begin{figure}[h]
	\centering
	\includegraphics[scale=0.6]{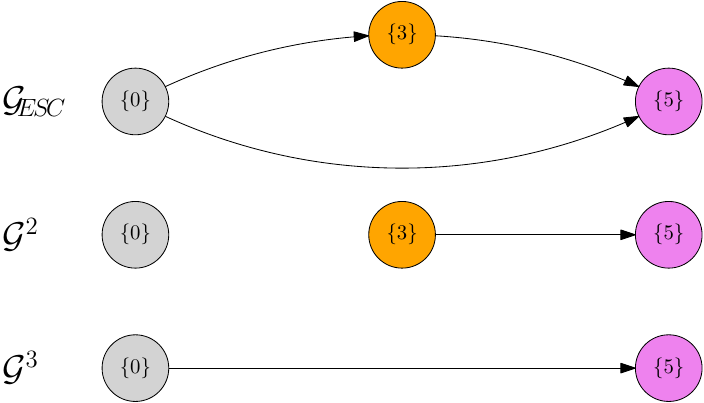}
	\caption{Metastability graph $\GG_\ESC$ and $L$-scale graphs $\GG^2$ and $\GG^3$ of Example \ref{ex:9_MultEscLowStab}}
	\label{fig:example7_graphs}
\end{figure}

%%%%%%%%%%%%%%%%%%%%%%%%%%%%%%%%%%%%%%%%%%%%%%%%%%%%%%%%%%%%%%%%%%%%%

\section{Proofs}
\label{sec:4_proofs}

In this chapter, we discuss the proofs of the results on metastable transitions and limiting jump processes that are presented in Sections \ref{sec:2.2_first_result} and \ref{sec:2.3_second_result} These build on the results in \cite{BoCoSm19} on the crossing of a fitness valley on a linear trait space and in \cite{CoKrSm21} on the faster $\ln K$-dynamics on general finite graphs. The main idea is to extend the techniques from \cite{BoCoSm19} to more complex trait spaces by considering sequential mutations along certain paths. Since mutations are very rare outside of the mutation spreading neighbourhood of the resident traits and unfit traits quickly go extinct, mutations along different paths can essentially be regarded as independent. Consequently, the overall rate of transitioning out of an ESC is obtained by summing over the rates of taking specific paths through the surrounding fitness valley.

The remaining chapter is structured as follows: In Section \ref{sec:4.1_proof_equilsize}, we determine the precise equilibrium size of the subpopulations with traits inside the mutation spreading neighbourhood. In Section \ref{sec:4.2_proof_path_rates}, we consider the rates at which mutants of any fitness arise along specific paths and combine these to the overall rate at which single mutants are born. Finally, in Section \ref{sec:4.3_proofs_one_step}, we combine these rates of producing mutants beyond the fitness valley with the probability of fixation and the faster $\ln K$-dynamics of reaching a new ESC to conclude Theorem \ref{Thm:main} and Corollary \ref{Cor:main}.
Section \ref{sec:4.4_proofs_multi_steps} is dedicated to the proof of Corollary \ref{Cor:kProcess} and Theorem \ref{Thm:Lscale}, where we concatenate several jumps across fitness valleys to obtain the multi-scale jump chain and carefully study which transitions are visible on the respective time scales to obtain the dynamics of the limiting Markov jump process.

\subsection[Equilibrium size]{Estimation of the equilibrium size}
\label{sec:4.1_proof_equilsize}

In this section we discuss the equilibrium population sizes of the living traits once an ESC is obtained. The results from \cite{CoKrSm21} only characterize the orders of population sizes $\b_w$ and the actual size $\bar{n}(\v)$ of the resident traits associated to an ESC. In order to calculate the precise transition rates from one ESC to another, we do, however, need a better estimate for the population sizes of the non-resident traits in $V_\alpha(\v)$.

We prove that, if the initial conditions of our process satisfy the assumptions of an asymptotic ESC, all living traits in $V_\alpha(\v)$ get arbitrarily close to their equilibrium size within a finite time. This equilibrium size preserves the orders of population sizes and is of the form
\begin{align}
		N^K_v(t)=a_v K\mu_K^{d(\v,v)} + o\left(K\mu_K^{d(\v,v)}\right) \qquad \forall v\in V_\a(\v),
\end{align}
for some $a_v\in \R_+$, which can be calculated precisely. The populations of living traits stay close to these equilibrium sizes as long as no new trait arises and reaches a size at which it can influence the population sizes of other traits, i.e.\ a size of order $K^{1/\alpha}$. To this extend, we recall the definition of the stopping time
\begin{align}\label{eq:T^K}
T^K_\fix:=\inf\dset{t\geq0:\exists w\in V\backslash V_\alpha(\v): \b^K_w(t)\geq1/\alpha}.
\end{align}

%%%% OLD VERSION %%%%
%\begin{lemma}[Equilibrium size inside the $\alpha$-radius]
%	\label{lem:equilibriumSize}
%	Let $\v\subset V$ and $(\b^K(0))_{K\geq0}$ be an asymptotic ESC.
%	Then, for all $\eps>0$, there exist constants $\t_\eps<\infty$, $U_\eps>0$ such that 
%	\begin{align}
%	\label{eq:equlibriumConv}
%	\lim_{K\to\infty}\Prob{\abs{\frac{N_v^K(t)}{K\mu_K^{d(\v,v)}}-a_v}<\eps \quad\forall t\in (\t_\eps,T^K\land \mathrm{e}^{U_\eps K}) \quad\forall v\in V_\a(\v)}=1,
%	\end{align}
%	where
%	\begin{align}
%	\label{eq:equilibriumSize}
%	a_v :=\sum_{\substack{\g:\v\to v\\ \abs{\g}=d(\v,v)}}
%	\bar{n}_{\g_0}(\v)
%	\prod_{i=1}^{\abs{\g}} \frac{b(\g_{i-1})m(\g_{i-1},\g_{i})}{\abs{f(\g_{i},\v)}}.
%	\end{align}
%\end{lemma}
%%%%%%%%%%%%%%%%%

\begin{lemma}[Equilibrium size inside the $\alpha$-radius]
	\label{lem:equilibriumSize}
	Let $\v\subset V$ and $(\b^K(0))_{K\geq0}$ be an asymptotic ESC.
	Then, for all $\eps>0$, there exist constants $\t_\eps<\infty$, $U_\eps>0$ and Markov processes $\left(N_v^{(K,\pm)}(t), t\geq 0\right)_{K\geq 0}$ such that,
	\begin{align}
	    \label{eq:equlibriumConv}
		\lim_{K\to\infty}\Prob{N_v^{(K,-)}\leq N_v^{K}\leq N_v^{(K,+)} \quad\forall t\in (\t_\eps,T^K_\fix\land \mathrm{e}^{U_\eps K}) \quad\forall v\in V_\a(\v)}=1 \nonumber\\
	\end{align}
	and
	\begin{align}
		\abs{\frac{\Exd{N_v^{(K,\pm)}(t)}}{K\mu_K^{d(\v,v)}}-a_v}<\eps \qquad\forall t\geq\t_\eps,
	\end{align}
	where
	\begin{align}
		\label{eq:equilibriumSize}
		a_v :=\sum_{\substack{\g:\v\to v\\ \abs{\g}=d(\v,v)}}
		\bar{n}_{\g_0}(\v)
		\prod_{i=1}^{\abs{\g}} \frac{b(\g_{i-1})m(\g_{i-1},\g_{i})}{\abs{f(\g_{i},\v)}}.
	\end{align}
\end{lemma}

\begin{proof}
	We will prove the claim by induction w.r.t.\ the distance from the resident traits. For the initialisation let us start with $v\in\v$. That is, we count also a single vertex as a path of length zero together with the convention that an empty product has the value one. In this case $(N^K_v,\ v\in\v)$ can be coupled with logistic birth-death processes with immigration, by estimating the incoming and outgoing mutants, which are of order $O(K\mu_K)$ or smaller. Hence we know already from \cite[Lemma A.6(ii)]{CoKrSm21} that the residents stabilize near their Lotka-Volterra-equlilibrium within a time of order $\OO(1)$. To make this more precise, define, for all $\eps>0$, the stopping time when the resident populations enter an $\eps$-neighbourhood of their equilibrium size
	\begin{align}
		\t_\eps^K:=\inf\dset{t\geq 0:\forall v\in\v: \abs{K^{-1}N^K_v(t)-\bar{n}_v(\v)}<\eps C}.
	\end{align}
	Here $C$ is a constant, depending only on the competition rates $c(v,w)$, which compensates the slight shift of the equilibrium due to small fluctuations of non-resident traits.
	Then there exists a constant time $\tilde{\t}_\eps<\infty$, such that $\lim_{K\to\infty}\Prob{\t_\eps^K<\tilde{\t}_\eps}=1$. After this time $\tilde{\t}_\eps$, the environment of competitive pressure stays almost constant, unless the fluctuations of the resident populations become too big or the non-residents reach a macroscopic level. These two events are described by the stopping times
	\begin{align}
		S^K_\eps:=\inf\dset{t\geq \t_\eps^K:\exists v\in\v: \abs{K^{-1}N^K_v(t)-\bar{n}_v(\v)}>2\eps C}
	\end{align}
	and
	\begin{align}
		\sigma^K_\eps:=\inf\dset{t\geq 0: \sum_{w\in V\backslash \v}N_w^K(t)\geq \eps K}.
	\end{align}
	We know from \cite[Propostition A.2]{ChMe11} that, for some constant $U_\eps>0$,
	\begin{align}
		\lim_{K\to \infty}\Prob{S_\eps^K>\mathrm{e}^{U_\eps K}\land\sigma^K_\eps}=1.
	\end{align}
	
	For the other traits in the $\a$-radius $v\in V_\a\backslash\v$ we prove as the induction step that \eqref{eq:equlibriumConv} is satisfied with 
	\begin{equation}
		\label{eq:equil_induction}
		a_v=\sum_{\substack{(w,v)\in E\\ d(\v,w)=d(\v,v)-1}}
		a_w\frac{b(w)m(w,v)}{\abs{f(v,\v)}}
	\end{equation}
	by deriving an upper and a lower bound on the population size through couplings. These bounds then immediately imply the claim.
	
	Following the notation of \cite{FoMe04}, we represent the population processes in terms of Poisson random measures. For this purpose let $(Q^{(b)}_v,Q^{(d)}_v,Q^{(m)}_{w,v};v,w,\in V)$ be independent homogeneous Poisson random measures on the space $(\R^2,\d s\d\theta)$ with constant rate one. Then we can write
	\begin{align}
		\label{eq:interiorPoissonRepr}
		N_v^K(t)
		=N_v^K(0)
		&+\int_0^t\int_{\R_+}\ifct{\theta\leq b(v)(1-\mu_K)N_v^K(s^-)}Q_{v}^{(b)}(\d s,\d\theta) \nonumber\\
		&-\int_0^t\int_{\R_+}\ifct{\theta\leq [d(v)+\sum_{w\in V}c^K(v,w) N_w^K(s^-)]N_v^K(s^-)} Q_{v}^{(d)}(\d s,\d\theta)\nonumber\\
		&+\sum_{(w,v)\in E} \int_0^t\int_{\R_+}\ifct{\theta\leq \mu_K b(w) m(w,v) N_w^K(s^-)}Q_{w,v}^{(m)}(\d s,\d\theta).
	\end{align}
	Note that we use the same Poisson measures to construct the processes for each $K$ here. However, as already pointed out in Section \ref{sec:2.1_model}, this is not necessary and we do not use any particular correlation between the processes for different $K$. We can use a specific joint construction here since we are only considering the convergence of probabilities of certain events, rather than of the processes themselves.
	
	Since we already know from \cite[Theorem 2.2]{CoKrSm21} that in the equilibrium state the non-resident populations $w\in V_\a(\v)$ stay of size $\OO(K\mu_K^{d(\v,w)})$, the main part of the mutations in the last line comes only from traits lying closer to the resident traits. Thus we can adopt the inductive structure of \cite[Lemma 7.1]{BoCoSm19} and approximate the population size of $v$ analogously by coupling it, for $K$ large enough, with two processes
	\begin{align}
		\label{eq:interiorCoupling}
		N^{(K,-)}_{v}(t)\leq N_v^K(t) \leq N^{(K,+)}_{v}(t),
		\quad \forall \tilde{\t}_\eps\leq t\leq \sigma_\eps^K \land T^K_\fix\land S^K_\eps.
	\end{align}

	To be precise, we take care of the admissible fluctuations of the residents by defining
	\begin{equation}
		\bar{n}_v^{(\pm)}(\v):=\bar{n}_v(\v)\pm2\eps C.
	\end{equation}
	Then, for $v\in V \backslash\v$ and $\mu_K<\eps$, we set
	\begin{align}
		&N^{(K,-)}_{v}(t)
		=N^K_v(\tilde{\t}_\eps)
		+\int_{\tilde{\t}_\eps}^t\int_{\R_+}\ifct{\theta\leq b(v)(1-\eps)N^{(K,-)}_{v}(s^-)}Q_{v}^{(b)}(\d s,\d\theta) \nonumber\\
		&\quad-\int_{\tilde{\t}_\eps}^t\int_{\R_+}\ifct{\theta\leq [d(v)+\sum_{w\in\v}c(v,w) \bar{n}_w^{(+)}(\v)+\eps\max_{\tilde{w}\in V\backslash\v}c(v,\tilde{w})]N^{(K,-)}_{v}(s^-)} Q_{v}^{(d)}(\d s,\d\theta)\nonumber\\
		&\quad+\sum_{(w,v)\in E} \int_{\tilde{\t}_\eps}^t\int_{\R_+}\ifct{\theta\leq \mu_K b(w) m(w,v) N^{K}_{w}(s^-)}Q_{w,v}^{(m)}(\d s,\d\theta)
	\end{align}
	and
	\begin{align}
		N^{(K,+)}_{v}(t)
		=&N^K_v(\tilde{\t}_\eps)
		+\int_{\tilde{\t}_\eps}^t\int_{\R_+}\ifct{\theta\leq b(v)N^{(K,+)}_{v}(s^-)}Q_{v}^{(b)}(\d s,\d\theta) \nonumber\\
		&-\int_{\tilde{\t}_\eps}^t\int_{\R_+}\ifct{\theta\leq [d(v)+\sum_{w\in\v}c(v,w) \bar{n}_w^{(-)}(\v)]N^{(K,+)}_{v}(s^-)} Q_{v}^{(d)}(\d s,\d\theta)\nonumber\\
		&+\sum_{(w,v)\in E} \int_{\tilde{\t}_\eps}^t\int_{\R_+}\ifct{\theta\leq \mu_K b(w) m(w,v) N^{K}_{w}(s^-)}Q_{w,v}^{(m)}(\d s,\d\theta),
	\end{align}
	where we use the same Poisson measures as in \eqref{eq:interiorPoissonRepr}. Note that this coupling satisfies \eqref{eq:interiorCoupling} only on the event $\dset{\t_{\eps}^K<\tilde{\t}_{\eps}}$. However, as mentioned above, this event's probability converges to $1$ and we can hence restrict our considerations to this case to obtain the desired convergence.

	On closer inspection, the approximating processes $N^{(K,-)}_{v},N^{(K,+)}_{v}$ are nothing but subcritical birth-death processes with immigration stemming form incoming mutations.
	
	Similar to the proof of \cite[Equation (7.8) et sqq.]{BoCoSm19} we can use the martingale decomposition of $N^{(K,+)}_v$ and $N^{(K,-)}_v$ to derive, for $t>\tilde{\t}_\eps$, the differential equation
	\begin{align}
		\frac{\d}{\d t}&\Exd{N^{(K,*)}_{v}(t)}\nonumber\\
		&=\left(b(v)(1-\ifct{\{*=-\}}\eps)-d(v)-\sum_{w\in\v}c(v,w)\bar{n}_w^{( \bar{*})}(\v)-\ifct{\{*=-\}}\eps\sup_{\tilde{w}\in V\backslash\v}c(v,\tilde{w})\right)\nonumber\\
		&\hspace{1em}\times\Exd{N^{(K,*)}_{v}(t)}
		+\sum_{(w,v)\in E} \mu_K b(w)m(w,v)\Exd{N^{K}_{w}(t)}\\
		&=f^{(*)}(v,\v)\Exd{N^{(K,*)}_{v}(t)}+\sum_{(w,v)\in E} \mu_K b(w)m(w,v)\Exd{N^{K}_{w}(t)},
	\end{align}
	where $\bar{*}=\{+,-\}\backslash*$ denotes the inverse sign.

	Here, we introduce $f^{(*)}(v,\v)$ as a short notation to point out that this is nothing but a perturbation of the invasion fitness.
	Then we can apply our a priori knowledge on the size of the sub-populations, i.e.
	\begin{equation}
		\Exd{N_w^K(t)}=\OO\left(K\mu_K^{\d(\v,w)}\right) \qquad \forall w\in V_\a(\v),
	\end{equation}
	to rewrite the ODE system
	\begin{align}
		\frac{\d}{\d t}\Exd{N^{(K,*)}_{v}(t)}
		=&\ f^{(*)}(v,\v)\Exd{N^{(K,*)}_{v}(t)}
		+\hspace{-1,5em}\sum_{\substack{(w,v)\in E\\d(\v,w)=d(\v,v)-1}} \mu_K b(w)m(w,v)\Exd{N^{K}_{w}(t)}\nonumber\\
		&+\OO\left(K\mu_K^{\d(\v,v)+1}\right)\\
		=&\ f^{(*)}(v,\v)\Exd{N^{(K,*)}_{v}(t)}
		+\hspace{-1,5em}\sum_{\substack{(w,v)\in E\\d(\v,w)=d(\v,v)-1}}  b(w)m(w,v)a_w K\mu_K^{\d(\v,v)}\nonumber\\
		&+o\left(K\mu_K^{\d(\v,v)}\right).
	\end{align}
	Here we use the induction hypothesis to estimate the populations with traits lying closer to the residents in the latter equality.
	
	Rescaling with $K\mu_K^{d(v,\v)}$ and using \eqref{eq:equil_induction}, the equation becomes
	\begin{align}
		\frac{\d}{\d t}\Exd{\frac{N^{(K,*)}_{v}(t)}{K\mu_K^{\d(\v,v)}}}
		&=f^{(*)}(v,\v)\Exd{\frac{N^{(K,*)}_{v}(t)}{K\mu_K^{\d(\v,v)}}}
		+a_v\abs{f(v,\v)} +o(1).
	\end{align}
	
	By variation of constants the solution is given by
	\begin{align}
		\label{eq:VarOfConst}
		\Exd{\frac{N^{(K,*)}_{v}(t)}{K\mu_K^{\d(\v,v)}}}
		=&\ \mathrm{e}^{f^{(*)}(v,\v)(t-\tilde{\t}_\eps)}
		\left(\Exd{\frac{N^{K}_{v}(\tilde{\t}_\eps)}{K\mu_K^{\d(\v,v)}}}-\frac{\abs{f(v,\v)}}{\abs{f^{(*)}(v,\v)}}a_v+o(1)\right)\nonumber\\
		&+\frac{\abs{f(v,\v)}}{\abs{f^{(*)}(v,\v)}}a_v+o(1)
		\nonumber\\
%		&=\mathrm{e}^{f^{(*)}(v,\v)(t-\tilde{\t}_\eps)}
%		\tilde{C}_\eps+(1\pm\eps\tilde{c}_\eps)a_v+o(1)
	\end{align}
	
	Note that the term in brackets can be bounded uniformly in $K$ and $\eps$, for $\eps$ small enough. Moreover the ratio of (perturbed) fitness can be expressed as $(1\pm\eps\tilde{c}_\eps)$. So \eqref{eq:VarOfConst} becomes
	\begin{align}
		\Exd{\frac{N^{(K,*)}_{v}(t)}{K\mu_K^{\d(\v,v)}}}=\mathrm{e}^{f^{(*)}(v,\v)(t-\tilde{\t}_\eps)}
		\OO(1)+(1\pm\eps\tilde{c}_\eps)a_v+o(1)
	\end{align}
	
	Finally taking into account that the fitness $f^{(*)}(v,\v)<0$ is negative for $v\in V_\a(\v)$ the first term vanishes for increasing time. Hence we see that for all $\tilde{\eps}>0$ there are $\eps>0$ and  $\t_{\tilde{\eps}}\in(\tilde{\t}_\eps,\infty)$ and $K_0\in\N$ such that, for all $t>\t_{\tilde{\eps}}$ and $K>K_0$
	\begin{align}
		\abs{\Exd{\frac{N^{(K,*)}_{v}(t)}{K\mu_K^{\d(\v,v)}}}-a_v}<\tilde{\eps}.
	\end{align}
	
	A posteriori, we then see that
	\begin{align}
		\lim_{K\to\infty}\Prob{\sigma_\eps^K<T^K_\fix\land\mathrm{e}^{U_\eps K}}=0,
	\end{align}
	which allows us to drop the stopping time $\sigma_\eps^K$ in the claim.
\end{proof}

\subsection[Path rates]{Pathwise evolution rates}
\label{sec:4.2_proof_path_rates}
From the precise description of the population sizes inside the mutation spreading neighbourhood we can now deduce the rate of occurrence of mutants that lay outside.

To observe a new mutant, whose trait is far away from the resident population, a whole sequence of mutation steps is needed. Traits outside the $\a$-neighbourhood $V_\a(\v)$ cannot avoid extinction only due to incoming mutants. Therefore, if such a trait has negative invasion fitness, mutants only give rise to small excursions approximated by subcritical branching processes. During each of these excursions there is a small probability that a new mutant is produced before extinction.

To overcome the problem of tracking possible back mutations, we not only observe the sizes of the different mutant populations. Instead, we distinguish mutants by the mutational path along which they arose and keep track of the genealogy. We set
\begin{equation}
	N_v^K(t)=\sum_{\g:\partial V_\alpha\to v} N_{v,\g}^K(t) \qquad\forall v\in V\backslash V_\a,
\end{equation}
where the pathwise mutations can by represented by
\begin{align}
	\label{eq:pathwisePoissonRepr}
	N_{v,\g}^K(t)
	&=\int_0^t\int_{\R_+}\ifct{\theta\leq b(v)(1-\mu_K)N_{v,\g}^K(s^-)}Q_{v,\g}^{(b)}(\d s,\d\theta) \nonumber\\
	&+\int_0^t\int_{\R_+}\ifct{\theta\leq \mu_K b(\tilde{v}) m(\tilde{v},v) N_{\tilde{v},\g\backslash v}^K(s^-)}Q_{v,\g}^{(m)}(\d s,\d\theta) \nonumber\\
	&-\int_0^t\int_{\R_+}\ifct{\theta\leq [d(v)+\sum_{w\in V}c^K(v,w) N_w^K(s^-)]N_{v,\g}^K(s^-)} Q_{v,\g}^{(d)}(\d s,\d\theta).
\end{align}
Here $\tilde{v}$ stands for the next-to-last vertex in $\g$, which is the progenitor of $v$ in $\g$, and for $\tilde{v}\in\partial V_\alpha$ we set
	\begin{align}
		N_{\tilde{v},\emptyset}(t):=N_{\tilde{v}}(t).
\end{align}
As before, $(Q^{(b)}_{v,\g},Q^{(d)}_{v,\g},Q^{(m)}_{v,\g};v\in V,\g:\partial V_\a\to v)$ are independent homogeneous Poisson random measures with constant rate one.

\begin{remark}\label{Rem:extinctoutside}
It suffices to only sum over the paths starting in $\partial V_\a$ in the decomposition. By the definition of $T^K_\ESC$ all populations outside of $V_\a$ are extinct at that time. The probability that a mutant of trait $v\in V\backslash V_\a$ arises before the finite time $\tau_\eps$ in Lemma \ref{lem:equilibriumSize}, when the populations in $V_\a$ reach their equilibrium, goes to zero. After this time we have good bounds on the population sizes of all traits in $V_\a$ and it is therefore sufficient to trace back the genealogy of new mutants to the last trait in $V_\a$, i.e.\ a trait in $\partial V_\a$.
\end{remark}

    With this representation at hand, we are now able to define the cumulated number of mutant individuals of trait $v$ that arose as mutants of the progenitor $\tilde{v}$, along the path $\g$
\begin{align}
	\label{eq:mutationCounter}
	M_{v,\g}^K(t)=\int_0^t\int_{\R_+}\ifct{\theta\leq \mu_K b(\tilde{v}) m(\tilde{v},v) N_{\tilde{v},\g\backslash v}^K(s^-)}Q_{v,\g}^{(m)}(\d s,\d\theta),
\end{align}
as well as the respective occurrence times of these mutants

\begin{align}
	\label{eq:pathwiseWaitingtime}
	T^{(i,K)}_{v,\g}:=\inf\left\{t\geq 0: M_{v,\g}^K(t)\geq i \right\},
\end{align}
where we set $T^{(0,K)}_{v,\g}:=0$.

Our aim is to show that new mutants outside of $V_\a$ appear at the end of a mutation path approximately as a Poisson point process with rate scaling with length of the path.

\begin{lemma}
	\label{lem:pathRates}
	Suppose $\v$ and $(\beta^K(0))_{K\geq0}$ are an asymptotic ESC and let $T^K_\fix$ be defined as in \eqref{eq:T^K}. Let $v\in V\backslash V_\a$ and $\g:\partial V_\a\to v$ be such that $\abs{\g}\geq L-\lalpha$ and $f(\g_i,\v)<0$, for all $i=0,\ldots,\abs{\g}-1$. Then there exist $0<c,C<\infty$ such that, for each $\eps>0$, there exist two Poisson point processes $M^{(K,\pm)}_{v,\g}$ with rates $\tilde{R}^{(\pm)}_{v,\g}K\mu_K^{\lalpha+\abs{ \g}}$ such that
	\begin{align}
		\liminf_{K\to\infty} \Prob{M^{(K,-)}_{v,\g}(t)<M_{v,\g}^K(t)<M^{(K,+)}_{v,\g}(t),\ \forall t< T^K_\fix}\geq 1-c\eps,
	\end{align}
	where the rate parameters are defined as
	\begin{align}
		\label{eq:pathwiseRates}
		\tilde{R}_ {v,\g}:=a_{\g_0} b(\g_0)m(\g_0,\g_1) \prod_{j=1}^{\abs{\g}-1} \lambda(\rho(\g_j,\v))m(\g_j,\g_{j+1}),
		&&\tilde{R}^{(\pm)}_{v,\g}=(1\pm C\eps)\tilde{R}_ {v,\g}.
	\end{align}
\end{lemma}

For the definitions of $\lambda(\rho)$ and $\rho(v,\v)$ we refer to \eqref{def:lambda} and \eqref{def:rho} respectively, while $a_{\g_0}$ is the equilibrium size defined in \eqref{eq:equilibriumSize}.

\begin{proof}
	Note that, throughout the whole proof, we assume that $\tau_\eps<t<T^K_\fix\land e^{U_\eps K}$, where $\tau_\eps$ and $U_\eps$ are defined in Lemma \ref{lem:equilibriumSize}. This can then be extended to all $0\leq t<T^K_\fix$ in the limit of $K\to\infty$ since $T^K_\fix<e^{U_\eps K}$ with probability converging to $1$ and, since $\mu_K\to0$, there is almost surely no mutation event during the finite time interval $[0,\tau_\eps]$.
	
	Let $v\in V\setminus V_\a$ and $\g:\partial V_\a\to v$ be given as in the Lemma. To better distinguish from the full path $\gamma$, we refer to the vertices of the path via $\g=(v_0,v_1,\ldots,v_{\abs{\g}})$. The idea of this proof is to considered the path isolated from the remaining graph and adapt the tools from \cite[Ch. 7.3.]{BoCoSm19} to the present situation. We refrain from adding much more notation to our already complicated situation. We try to handle the far more general structure of our trait graph by translating the notation of the central objects between the articles instead.
	
	The first observation is that, for every $t<T^K_\fix$, we can bound the mutant counting process of trait $v_1$ by
	\begin{align}
		M_{v_1,\g}^{(K,-)}(t)
		\leq M_{v_1,\g}^K(t)
		\leq M_{v_1,\g}^{(K,+)}(t)
		\quad\text{a.s.},
	\end{align}
	with the bounding processes being defined as
	\begin{align}
			M_{v_1,\g}^{(K,\pm)}(t)=\int_0^t\int_{\R_+}\ifct{\theta\leq \mu_K b(v_0) m(v_0,v_1) N_{v_0}^{(K,\pm)}(s^-)}Q_{v_0,\g}^{(m)}(\d s,\d\theta).
	\end{align}
	Note that the estimate corresponds to equation (7.42) in \cite{BoCoSm19}, while the definition is the adapted version of (7.17) therein. In order make use of Lemma \ref{lem:equilibriumSize}, we continue temporarily with the simplified processes
	\begin{align}
		\bar{M}_{v_1,\g}^{(K,\pm)}(t)=\int_0^t\int_{\R_+}\ifct{\theta\leq \mu_K b(v_0) m(v_0,v_1) \Exd{N_{v_0}^{(K,\pm)}(s^-)}}Q_{v_0,\g}^{(m)}(\d s,\d\theta)
	\end{align}
	and
	\begin{align}
		\bar{T}^{(i,K,\pm)}_{v_1,\g}:=\inf\left\{t\geq 0: \bar{M}_{v_1,\g}^{(K,\pm)}(t)\geq i \right\}.
	\end{align}

	In fact, this turns out to be sufficient for our results since a standard application of Doob's martingale inequality shows that, with probability converging to 1, the difference of the processes $M_{v_1,\g}^{(K,\pm)}$ and $\bar{M}_{v_1,\g}^{(K,\pm)}$ during the relevant time interval stays of sufficiently small order. To be precise there exist sequences of numbers $N_1(K)$ and $N_2(K)$, with
	\begin{align}
		N_1(K)\gg (K\mu_K^{L})^{-1}
		\quad\text{and}\quad
		N_2(K)\ll (\mu_K^{L-1-\lalpha})^{-1}
	\end{align}
	such that
	\begin{align}
		\lim_{K\to\infty} \Prob{\sup_{s\leq N_1(K)}\abs{M_{v_1,\g}^{(K,\pm)}(s)-\bar{M}_{v_1,\g}^{(K,\pm)}(s)}>N_2(K)}=0.
	\end{align}
	For details, see \cite[p.3583]{BoCoSm19}. At each time $\bar{T}^{(i,K,\pm)}_{v_1,\g}$ an individual of trait $v_1$ is born. In order to track its ancestors until potentially a trait $v_{\abs{\g}}$ individual is born, in a similar way as done in the previous section, we couple the $k$-mutant population, for $\lalpha+1\leq k\leq \abs{\g}-1$, to birth-death processes with individual birth and death rates
	\begin{align}
		b^{(*)}(v_k)&=b(v_k)(1-\ifct{\{*=-\}}\eps)\\ d^{(*)}(v_k)&=d(v_k)+\sum_{w\in\v}c(v_k,w)\bar{n}_w^{(\bar{*})}(\v)+\ifct{\{*=-\}}\eps\sup_{\tilde{w}\in V\backslash\v}c(v_k,\tilde{w}).
	\end{align} Note that in contrast to Section \ref{sec:4.1_proof_equilsize}, these processes are subcritical and hence go extinct in finite time. However, there is a small probability during such an excursion of the $k$-mutant population that an individual of trait $(k+1)$ is born.
	Analogously to \cite[pp. 3581-3582]{BoCoSm19}, we can use Lemma \ref{lem:offspringOfAnExcursion} (see Appendix \ref{app:A1}) to derive
	\begin{multline}
		\Prob{\text{An excursion of trait $v_k$ produces exactly 1 mutant of type $v_{k+1}$}}\\
		=\mu_K\lambda(\rho(v_k,\v))m(v_k,v_{k+1})(1+\OO(\eps)),
	\end{multline}
		while on the other hand
	\begin{align}
		\Prob{\text{An excursion of trait $v_k$ produces at least 2 mutants of $v_{k+1}$}}=\OO(\mu_K^2).
	\end{align}
	Hence, the probability that the $i$-th mutant of trait $v_1$ (i.e.\ the one triggering $\bar{T}^{(i,K,\pm)}_{v_1,\g}$) produces a $v_{\abs{\g}}$-mutant is, for large $K$,
	\begin{align}
	\label{eq:ThinningProb}
		\mu_K^{\abs{\g}-1}
		\left(\prod_{k=\lalpha+1}^{\abs{\g}-1} \lambda(\rho(v_k,\v))m(v_k,v_{k+1}) \right)
		(1+\OO(\eps)).
	\end{align}

	Since Lemma \ref{lem:equilibriumSize} implies that $\bar{M}_{v_1,\g}^{(K,\pm)}$ can be treated as a Poison process with intensity
	\begin{align}
		K\mu_K^{d(\v, v_0)+1} a_v b(v_0) m(v_0,v_1),
	\end{align}
	we get appearance of $v_{\abs{\g}}$-mutants also as Poison process with thinned intensity
	\begin{align}
	&K\mu_K^{d(\v, v_0)+\abs{\g}}
	a_v b(v_0) m(v_0,v_1)
	\left(\prod_{k=\lalpha+1}^{\abs{\g}-1} \lambda(\rho(v_k,\v))m(v_k,v_{k+1}) \right)
	(1+\OO(\eps))\\
	&=\tilde{R}^{(\pm)}_{v,\g}K\mu_K^{\lalpha+\abs{ \g}}.
	\end{align}

	Eventually, the difference between $M_{v_1,\g}^{(K,\pm)}$ and $\bar{M}_{v_1,\g}^{(K,\pm)}$ is of smaller order than $(\mu_K^{L-1-\lalpha})^{-1}$ and multiplying with the thinning probability \eqref{eq:ThinningProb}, which is of order $\mu_K^{\abs{\g}-1-\lalpha}$, this only changes the appearance rate for the $v_{\abs{\g}}$-mutants by a  vanishing order.
\end{proof}

\begin{remark}
	\label{rem:Excursion}
	Note that in general there could be an overlap of two excursions of $N_{v_k,\g}^K$, associated to different incoming mutants. Nevertheless in the limit of $K\to\infty$ this does not happen since the time interval between the incoming mutants diverges, while the durations of the excursions stay of order one, i.e. $T_{v_k,\g}^{(i+1,K)}-T_{v_k,\g}^{(i,K)}\gg1$.
\end{remark}

As a direct corollary we can deduce the law of the appearance times of new mutants with trait $v\in V\backslash V_\a$.

\begin{corollary}
	\label{cor:mutantArrivals}
	Suppose $\v$ and $(\beta^K(0))_{K\geq0}$ are an asymptotic ESC. Let $v\in V\backslash V_\a$ be a trait such that all paths $\g:\partial V_\a\to v$ of shortest length $\abs{\g}=d(V_\a,v)$ do only visit traits with negative invasion fitness, excluding the last trait $v$, i.e.\ $f(\g_{i},\v)<0\ \linebreak \forall i=0,\ldots,\abs{\g}-1$. Denote  by $T^{(i)}_v$ the appearance time of the $i$-th mutant of trait $v$ descended from an nearest neighbour trait. Then there exists a $0<c<\infty$ such that, for each $\eps>0$, there exist sequences of iid. exponential random variables $E^{(i,\pm)}_v$, $i\geq 1$ with rates $\tilde{R}^{(\pm)}_v=(1\pm C\eps)\tilde{R}_v$, where
	\begin{align}
		\label{eq:traitRates}
		\tilde{R}_v:=\sum_{\substack{\g:\partial V_\a\to v\\ \abs{\g}=d(V_\a,v)}}a_{\g_0} b(\g_0)m(\g_0,\g_1) \prod_{j=1}^{\abs{\g}-1} \lambda(\rho(\g_j,\v))m(\g_j,\g_{j+1})
	\end{align}
	Such that
	\begin{align}
		\liminf_{K\to\infty} \Prob{E^{(i,-)}\leq K\mu_K^{d(\v,v)}\left(T^{(i)}_v-T^{(i-1)}_v\right)\leq E^{(i,+)} \left\vert T^{(i)}_v<T^K_\fix\right.}\geq 1-c\eps
	\end{align}
\end{corollary}

\begin{proof}
	Due to Lemma \ref{lem:pathRates}, we can describe the arrivals of new $v$-type mutants approximately as sum of Poisson point processes. Since the Poisson measures $Q^{(\cdot)}_{\cdot,\cdot}$ in our representation \eqref{eq:pathwisePoissonRepr} are taken as independent, the resulting mutation counting processes $M_{v,\g}^K$ are also independent. Hence their sum can be approximated by a Poisson process with with intensity
	\begin{align}
	 \sum_{\g:\partial V_\a\to v} \tilde{R}_{v,\g}K\mu_K^{\abs{\g}+\lalpha}. 
    \end{align}	
	
	Since each summand scales with the length of the respective path, the first order of the overall rate is given only by the shortest paths (i.e.\ $\g$ such that $\abs{\g}=d(V_\a,v)\linebreak[1]=d(\v,v)-\lalpha$). As a result, the first order becomes \eqref{eq:traitRates} multiplied by $K\mu^{d(\v,v)}$. Finally, the waiting times of homogeneous Poisson point processes are exponentially distributed with the same rate.
\end{proof}

\subsection{Proof of Theorem \ref{Thm:main} and Corollary \ref{Cor:main}}
\label{sec:4.3_proofs_one_step}

We have now assembled all the tools to finish the proof of Theorem \ref{Thm:main} and Corollary \ref{Cor:main}.

Note that, with the notation from the proof of Lemma \ref{lem:equilibriumSize}, all following considerations are only valid up to the stopping time $S^K_\eps\land\sigma^K_\eps$, for sufficiently small $\eps$.  Since we have seen previously that $T^K_\fix\leq S^K_\eps\land\sigma^K_\eps$ with probability converging to one, as $K\to\infty$, we do not condition on this anymore in the following. Moreover, constants $c$ and $C$ may vary throughout the proof but are always assumed to satisfy $0<c,C<\infty$.

Both results assume that the initial conditions $(\beta^K(0))_{K\geq0}$ compose an asymptotic ESC associated to the coexisting traits $\v\subset V$. In a first step, we study the time until the fixation of the first mutant trait outside of $V_\a:=V_\a(\v)$, i.e.\ $T^K_\fix$. Corollary \ref{cor:mutantArrivals} implies that, for all traits $w\in V\backslash V_\alpha$ such that all shortest paths $\g:\v\to v$ only pass through unfit traits, new mutants of this trait arise approximately according to a Poisson point process with rate $\tilde{R}_v$. By assumption,  $\b^K_w(0)=0$, for all $K>K_0$ and $w\in V\backslash V_\a$, i.e.\ all traits outside of $V_\a$ are initially extinct. As a result, individuals of such traits $w$ are only present due to the above incoming mutations.

We now argue why it suffices to consider traits $w\in V\backslash V_\a$ such that $f(w,\v)>0$ and $d(\v,w)=L(\v)$, i.e.\ the $w\in V_\mut:=V_\mut(\v)$, as candidates to reach $\b^K_w=1/\a$ first and trigger the stopping time $T^K_\fix$.

For all $w$ such that $\lalpha<d(\v,w)<L(\v)$, the definition of $L(\v)$ yields $f(w,\v)<0$. Therefore, the descendants of a mutant of such traits can be bounded from above by a subcritical birth-death process with rates that do not depend on $K$, that dies out within a finite time with probability 1. As a result,
\begin{align}
\lim_{K\to\infty}\P\left(\sup_{t\in[0, T^K_\fix\land e^{U_\eps K}]}\b^K_w(t)\geq\frac{1}{\a}\right)=0.
\end{align}

For $w$ such that $d(\v,w)=L(\v)$ and $f(w,\v)<0$, the same argument can be applied.

Finally, for all $w$ such that $d(\v,w)>L(\v)$, for all $T<\infty$, Corollary \ref{cor:mutantArrivals} implies that the arrival time of the first $w$ mutant, $T^{(1)}_w$, satisfies
\begin{align}
\lim_{K\to\infty}\P\left(T^{(1)}_w\leq \frac{T}{K\mu_K^{L(\v)}}\land T^K_\fix\right)=0.
\end{align}

Focussing on the $w \in V_\mut$, we can use couplings to supercritical birth-death processes (similar to the arguments in the previous sections) to bound the different mutant populations. Using classical results on branching processes (e.g.\ from \cite[Ch.\ III.4]{AthNey72}) we can approximate the probability that the descendants of a single mutant of a particular trait $w$ do not go extinct by $(1\pm C\eps)f(w,\v)/b(w)$. Moreover, conditioned on not going extinct, the time that such a population needs to grow to a size of $K^{1/\a}$ can be bounded by $(1\pm C\eps)\ln K/\a f(w,\v)$. It is therefore negligible on the time scale $1/K\mu_K^{L(\v)}$, on which the $w$ mutants arise.

Overall, we can deduce from Corollary \ref{cor:mutantArrivals} that there is a constants $0<c<\infty$ and exponential random variables $E^\pm_{w,\fix}$ with parameters $(1\pm c\eps)\tilde{R}_wf(w,\v)/b(w)=(1\pm c\eps)R(\v,w)$ such that
\begin{align}
	\liminf_{K\to\infty}\P\left(E^-_{w,\fix}\leq K\mu_K^{L(\v)}T^K_\fix\leq E^+_{w,\fix}\Big\vert\beta^K_w(T^K_\fix)=\frac{1}{\a}\right)\geq 1-c\eps.
\end{align}

Since the mutants arising along different paths are independent (see the proof of Corollary \ref{cor:mutantArrivals}), the actual stopping time $K\mu_K^{L(\v))}T^K_\fix$ (without conditioning on a trait $w$) is roughly exponentially distributed with the sum of all rates $R(\v)=\sum_{w\in V_\mut}R(\v,w)$.  In addition, the probability that a certain trait $w\in V_\mut$ triggers the stopping time $T^K_\fix$ can be approximated by $R(\v,w)/R(\v)$.  More precisely, there are exponential random variables $E^\pm(\eps)$ such that
\begin{align}
\liminf_{K\to\infty}\P\left(E^-(\eps)\leq K\mu_K^{L(\v)}T^K_\fix\leq E^+(\eps)\right)\geq 1-c\eps,\\
\frac{R(\v,w)}{R(\v)}(1-c\eps)\leq\lim_{K\to\infty}\P\left(\beta^K_w(T^K_\fix)=\frac{1}{\a}\right)\leq\frac{R(\v,w)}{R(\v)}(1+c\eps).
\end{align}

Since $\eps$ can be picked arbitrarily small, this concludes the proof of Theorem \ref{Thm:main}.

To deduce Corollary \ref{Cor:main}, we note that at time $T^K_\fix$ the population sizes satisfy \eqref{eq:afterFix}, for some $w\in V_\mut(\v)$. Hence the assumption of the corollary and Theorem \ref{Thm_betalogK} imply that a new ESC associated to $\v_\ESC (\v,w)$ is obtained within a time of order $\ln K$. We emphasise that, although Theorem \ref{Thm_betalogK} only implies that $\beta^K_u\to0$ for traits\linebreak $u\notin V_\a(\v_\ESC (\v,w))$ after this time, these subpopulations can be bounded from above by subcritical branching processes that go extinct within a time of order 1, such that the conditions of $T^K_\ESC$ are truly satisfied. This yields the first claim of Corollary \ref{Cor:main}. Since this time is again negligible with respect to the $1/K\mu_K^{L(\v)}$-time scale, the second claim follows directly.  For the last claim, we realise that a new ESC $\w$ might be reached from multiple $w\in V_\mut(\v)$, and we therefore add up all corresponding probabilities to obtain $p(\v,\w)$. This concludes the proof of Corollary \ref{Cor:main}.

\subsection{Proof of Corollary \ref{Cor:kProcess} and Theorem \ref{Thm:Lscale}}
\label{sec:4.4_proofs_multi_steps}

In order to derive results for the jump chain $(\v^{(k)})_{k\geq0}$ on $\GG_\ESC$, we observe that, after a successful transition according to Corollary \ref{Cor:main}, the final state of the process again satisfies the initial assumptions for another application of the corollary. We simply need to recompute the state-dependent quantities ($L(\v),V_\mut(\v)$, etc.). As a consequence, the strong Markov property allows us to use Corollary \ref{Cor:main} to construct the random sequence $(\v^{(k)})_{k\geq 0}$ as well as derive the asymptotics of the stopping times $T_\ESC^{(k,K)}$ by an inductive procedure. This proves Corollary \ref{Cor:kProcess}.

To extract the limiting process on the time scale $1/K\mu_K^L$ for fixed $L>\alpha$, take an initial configuration of this stability degree, i.e.\ $\v\in\SS^L$. Considering the jump chain $(\v^{(k)})_{k\geq0}$ with $\v^{(0)}=\v$, Assumption \ref{Ass:NoCycles} implies that, with probability one, $(\v^{(k)})_{k\geq0}$ reaches an ESC of stability degree at least $L$ within finitely many steps. We now consider such a finite path $\G:\v\to\w$ in $\GG_\ESC$, where $L(\w)\geq L$. Without loss of generality we may assume that the intermediate ESCs are of strictly lower stability degree, i.e.\ $L(\G_i)<L\linebreak \forall 1\leq i<\abs\G$. Otherwise we could shorten the path. Asking now for the time $T_\G^K$ that it takes to transition from $\v$ to $\w$ along $\G$, we can simply add up the single step transition times $T^{(i,K)}-T^{(i-1,K)}$. By Corollary \ref{Cor:kProcess}, we know that, on the time scale $1/K\mu_K^{L^{(i)}}$, those are well approximated by exponential random variables $E^{(i)}_\pm$. Since $L=L^{(1)}>L^{(i)}$, for $2\leq i\leq\abs\G$, we can deduce that the rescaled transition time $T_\G^K K\mu_K^L$ is dominated by the very first transition and thus well described by exponential random variables.

To compute the respective transition rates, notice that by Corollary \ref{Cor:kProcess}, on the time scale $1/K\mu_K^L$, the rate to escape from $\v=\G_0$ is given by $R(\v)=R^{(1)}$. Moreover, we have to take into account that we consider the case where the limit process $(\v^{(i)})_{i\geq i}$ takes a particular path, i.e.\ $\v^{(i)}=\G_i$, for $0\leq i\leq\abs\G$. The probability of this event is simply given by the product of the one-step-probabilities $p(\v^ {(i-1)},\v^{(i)})$. Similarly to previous arguments, there might by different paths $\G:\v\to\w$ and hence we add up their probabilities. This yields the rates $R^L(\v,\w)$ in \eqref{eq:RatesLTimescale} and therefore the claimed dynamics of the jump process $(\v^L(t))_{t\in[0,T]}$ on the $L$-scale graph $\GG^L$.

To finally deduce the limit of the rescaled population process $N^K/K$, we note that there is no macroscopic evolution during almost the entire waiting time for a transition on $\GG^L$. The set of macroscopic traits $\dset{v\in V:\b^K_v(t)>1-\eps_K}$ only changes after a new mutant fixates, which happens at time $T^{(1,K)}_\fix$. The rest of the transition time, which may consist of many chances of the macroscopic traits, vanishes when rescaling with $K\mu_K^L$. Therefore, we obtain the limit process of Theorem \ref{Thm:Lscale}, which jumps between the Lotka-Volterra-equilibria associated to the state of $(\v^L(t))_{t\in[0,T]}$.

%%%%%%%%%%%%%%%%%%%%%%%%%%%%%%%%%%%%%%%%%%%%%%%%%%%%%%%%%%%%%%%%%%%%%

%\backmatter{}

\appendix

\section{Technical results}
\label{appendix}
The aim of this chapter is to collect some results on the $\OO(1)$- and $\OO(\ln K)$-time scale behaviour of the population process. While Section \ref{app:A1} explains the form of $\lambda(\rho)$, Section \ref{app:A2} justifies the notation $\v_\ESC (\v,v)$. The statements have been derived in \cite{BoCoSm19} and \cite{CoKrSm21} whereto we refer for detailed proofs.

\subsection{Excursions of subcritical birth death processes}
\label{app:A1}
The first lemma quantifies the mean number of birth events before a subcritical birth death process goes extinct, corresponding to $\lambda(\rho)$. Although we restate an existing result here, we provide a short proof below. This proof is different to the more general scenario that is cited in \cite{BoCoSm19} and gives the reader an intuition behind the expression.
\begin{lemma}
	\label{lem:offspringOfAnExcursion} (\cite[Lemma A.3]{BoCoSm19})
	Consider a subcritical linear birth death process with individual birth and death rates $0<b<d$. Denote by $Z$ the total number of birth events during an excursion of this process initiated with exactly one individual. Then, for $k\in\N_0$,
	\begin{align}
		\label{eq:ExcursionProb}
		p^{(b,d)}(k):=\Prob{Z=k}=\frac{(2k)!}{k!(k+1)!}\left(\frac{b}{b+d}\right)^k\left(\frac{d}{b+d}\right)^{k+1}
	\end{align}
	and in particular
	\begin{align}
		\label{eq:ExcursionMean}
		e^{(b,d)}:=\Exd{Z}=\sum_{k=1}^{\infty}\frac{(2k)!}{(k-1)!(k+1)!}\left(\frac{b}{b+d}\right)^k\left(\frac{d}{b+d}\right)^{k+1}.
	\end{align}
	Moreover we have the following continuity result. There exist two positive constants $c,\eps_0>0$, such that, for all $0<\eps<\eps_0$ and $0<b_i<d_i$, if $\abs{b_1-b_2}<\eps$ and $\abs{d_1-d_2}<\eps$, then
	\begin{align}
		\abs{e^{(b_1,d_1)}-e^{(b_2,d_2)}}<c\eps.
	\end{align}
\end{lemma}
\begin{remark}
	Note that \eqref{eq:ExcursionMean} corresponds to \eqref{def:rho} via $e^{(b,d)}=\lambda(\rho)$, where $\rho=b/(b+d)$.
\end{remark}

\begin{proof}
	Although the considered process takes place in continuous time, it suffices to focus on the birth and death events as jump chain in discrete time. This is nothing but a simple random walk on $\N_0$ with probabilities
	\begin{align}
		p(x,x+1)=\frac{b}{b+d},\quad p(x,x-1)=\frac{d}{b+d} \quad \forall x\geq 1
	\end{align}
	and absorbing state $0$.	From this point of view it is only a question of counting the number of paths leading from one individual to extinction consisting of exactly $k$ births  and hence $k+1$ death events. As final step there has to happen a death since the population does not vanish before. So the first $2k$ events form a walk from $1$ to $1$. There are $\binom{2k}{k}$ of such paths but some of them would lead to extinction earlier. To determine their number we apply a reflection principle in the following way. Let $x=(x_0,x_1,\ldots,x_{2k})$ be a path leading from one to one such that there exists a $0<j<2k$ with $x_j=0$. Then we define the partially reflected path $\tilde{x}$ by
	\begin{align}
		\tilde{x}_i:=\left\{ \begin{array}{rl}
			x_i & \text{for } i\leq j\\
			-x_i & \text{for }  i>j
		\end{array} \right.
	\end{align}
	This gives us a unique path from $\tilde{x}_0=1$ to $\tilde{x}_{2k}=-1$ (cf.\ \ref{fig:RelfectionPrinciple}). Moreover there is a one to one correspondence between prematurely extincting processes and paths leading from $1$ to $-1$. The latter ones consist of only $k-1$ births and hence there are $\binom{2k}{k-1}$ different ones.
	Finally the total number of legal paths is
	\begin{align}
		\#\dset{x=(x_0,x_1,\ldots,x_{2k})\vert x_0=1,x_{2k}=1,x_i>0}=\binom{2k}{k}-\binom{2k}{k-1}=\frac{(2k)!}{k!(k+1)!}.
	\end{align}
	We now achieve \eqref{eq:ExcursionProb} by multiplying with the probability of $k$ births and $k+1$ death events. The last statement is a simple consequence of the mean value theorem.
\end{proof}

\begin{figure}
	\centering
	\includegraphics[scale=0.8]{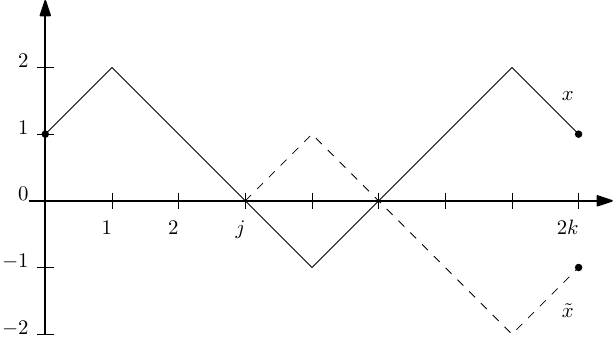}
	\caption{Original path $x$ that prematurely goes extinct and its reflection $\tilde{x}$.}
	\label{fig:RelfectionPrinciple}
\end{figure}

\subsection{Fast evolution until ESC}\label{proof_logK}
\label{app:A2}

In this subsection we discuss the first phase of evolution, where an ESC is obtained on the $\ln K$-time scale. The convergence of $N^K(t\ln K)/K$ and $\beta^K(t\ln K)$, as $K\to\infty$, is studied in \cite{CoKrSm21}. In the following we cite the respective results in the notation of this paper.
	
	For a finite graph $\GG=(V,E)$ and under Assumptions \ref{Ass_selfcomp} and \ref{Ass_alphafit}, the trajectories $(\beta_w(t), w\in V)$ (which turn out to be the limit of $(\beta^K_w(t\ln K), w\in V)$) are defined by an inductive procedure. The construction is valid until a stopping time $T_0$.
	
	Denote by $\tilde{\v}^{(\ell)}$, $\ell\geq0$, the sequence of consecutive coexisting resident traits. We emphasize that these are not to be confused with the sequence of resident traits $\v^{(k)}$, $k\geq0$, that are associated to ESCs. The invasion times, at which the sets of resident traits change due to upcoming mutant traits, are denoted by the increasing sequence $(s_\ell)_{\ell\geq0}$.
	
	For initial conditions $\tilde{\beta}(0)$, the support of the unique asymptotically stable equilibrium of the Lotka-Volterra system \eqref{LV_system} associated to the traits $\{w\in V:\tilde{\beta}_w(0)=1\}$ (if existent) is denoted by $\tilde{\v}^{(0)}$, i.e.\ $\bar{n}(\tilde{\v}^{(0)})=LVE_+(\tilde{\v}^{(0)})\subset LVE(\{w\in V:\tilde{\beta}_w(0)=1\})$. The equilibrium $\bar{n}(\tilde{\v}^{(0)})$ is reached within a time of order 1 and we set $s_0:=0$. Moreover, we define $\beta_w(0):=\max_{u\in V}[\tilde{\beta}_u(0)-d(u,w)/\alpha]_+$ as the initial condition of the limiting trajectories. This reflects that, within a time of order 1, living traits produce neighbouring mutant populations with the size of a $\mu_K$-fraction of their own size. This time of order 1 is negligible on the $\ln K$-time scale, which the limit $\beta$ is defined on.
	
	Assuming that $s_{\ell-1}$, $\tilde{\v}^{(\ell-1)}$ such that $LVE_+(\tilde{\v}^{(\ell-1)})=\bar{n}(\tilde{\v}^{(\ell-1)})$, and $\beta(s_{\ell-1})$ are known, the next phase can be described as follows. The $\ell^\text{th}$ invasion time is set to
	\begin{align}
		s_\ell:=\inf\{t>s_{\ell-1}:\exists\ w\notin\tilde{\v}^{(\ell-1)}:\beta_w(t)=1\}.
	\end{align}
	For $s_{\ell-1}\leq t\leq s_\ell$, for any $w\in V$, $\beta_w(t)$ is defined by
	\begin{align}\label{beta}
		\beta_w(t):=\max_{\substack{u\in V}}\left[\beta_u(s_{\ell-1})+(t-t_{u,\ell}\land t)f(u,\tilde{\v}^{(\ell-1)})-\frac{d(u,w)}{\alpha}\right]\lor 0,
	\end{align}
	where, for any $w \in V$,
	\begin{align} \label{deftwk}
		t_{w,\ell}:=\begin{cases}\inf\left\{t\geq s_{\ell-1}:\exists\ u\in V: d(u,w)=1, \beta_u(t)=\frac{1}{\alpha}\right\}
			&\text{if }\beta_w(s_{\ell-1})=0\\s_{\ell-1}&\text{else}\end{cases}
	\end{align}
	is the first time in $[s_{\ell-1},s_\ell]$ when this trait arises. If we define $V_\liv(t):=\{w\in V:\beta_w(t)>0\}$ equivalently to $V^K_\liv$ (on the $\ln K$-time scale), then this implies $\beta_w(t_{w,\ell})\geq0$ and $\beta_w(t_{w,\ell}+\delta)>0$, for small $\delta>0$.
	
	The stopping time $T_0$, that terminates the inductive construction of the limiting trajectories, is set to $s_\ell$ if
	\begin{itemize}
		\item[(a)] there is more than one $w\in V\backslash\tilde{\v}^{(\ell-1)}$ such that $\beta_w(s_\ell)=1$;
		\item[(b)] the mutation-free Lotka-Volterra system associated to $\tilde{\v}^{(\ell-1)}$ and the unique $w\in V\backslash\tilde{\v}^{(\ell-1)}$ such that $\beta_w(s_\ell)=1$ does not have a unique globally attractive stable equilibrium (in particular, if such an equilibrium does not exist for $\{w\in V:\tilde{\beta}_w(0)=1\}$, $T_0$ is set to 0);
		\item[(c)] there exists $w\in V\backslash\tilde{\v}^{(\ell-1)}$ such that $\beta_w(s_\ell)=0$ and $\beta_w(s_\ell-\delta)>0$ for all $\delta>0$ small enough. 
		\item[(d)] there exists $w\in V\backslash\tilde{\v}^{(\ell-1)}$ such that $s_\ell=t_{w,\ell}$.
	\end{itemize}
	These conditions are mostly technical and are discussed in \cite{CoKrSm21}.
	
	With this construction, the results can be stated as follows:
	
	\begin{theorem}\label{Thm_betalogK} (\cite[Theorem 2.7]{CoKrSm21})
		Let $\GG=(V,E)$ be a finite graph. Suppose that Assumption \ref{Ass_selfcomp} and \ref{Ass_alphafit} hold and consider the model defined by \eqref{Generator} with $\mu_K=K^{-1/\alpha}$.
		Let $\tilde{\v}_0\subset V$ and assume that, for every $w\in V$ ,
		\begin{align}\label{cond-init}
			\beta^K_w(0)\rightarrow\tilde{\beta}_w(0), \quad (K\to\infty) \quad \text{in probability}.
		\end{align}
		Then, for all $T>0$, as $K\to\infty$, the sequence $((\beta^K_w(t\ln K),w\in V),t\in[0,T\land T_0])$ converges in probability in 
		$\mathbb{D}([0,T\land T_0],\R_+^V)$ to the deterministic, piecewise affine, continuous function $((\beta_w(t),w\in V),t\in[0,T\land T_0])$, which is defined in \eqref{beta}.
	\end{theorem}
	
	\begin{theorem}\label{Thm_NlogK} (\cite[Proposition 2.8]{CoKrSm21})
		 Under the same assumptions as in Theorem \ref{Thm_betalogK}, 
		for all $T>0$, as $K\to\infty$, the sequence $((N^K_w(t\ln K)/K,w\in V),t\in[0,T\land T_0])$ 
		converges in the sense of the finite dimensional distributions
		to a deterministic jump process $((N_w(t),w\in V),t\in[0,T\land T_0])$, which jumps between different Lotka-Volterra equilibria according to
		\begin{align}
			N_w(t):=\sum_{\ell\in\N:s_{\ell+1}\leq T_0}\mathbf{1}_{s_\ell\leq t<s_{\ell+1}}\mathbf{1}_{w\in\tilde{\v}^{(\ell)}}\bar{n}_w(\tilde{\v}^{(\ell)}).
		\end{align}
	\end{theorem}

	Moreover, the invasion times $s_\ell$ and the times $t_{w,\ell}$ when new mutants arise are calculated precisely in \cite{CoKrSm21}. This is however not relevant to the discussion in this paper.
	
	We notice that the constructed trajectories $(\beta_w(t), w\in V)$ stay constant precisely once an ESC is obtained. In this case, there is no more visible evolution on the $\ln K$-time scale.

%%===========================================================================================%%

\bibliographystyle{abbrv}
\bibliography{Biomathe_bib}

\begin{thebibliography}{10}

\bibitem{AthNey72}
K.~B. Athreya and P.~E. Ney.
\newblock {\em Branching processes}.
\newblock Die Grundlehren der mathematischen Wissenschaften, Vol.\ 196.
  Springer-Verlag, New York-Heidelberg, 1972.

\bibitem{BaBoCh17}
M.~Baar, A.~Bovier, and N.~Champagnat.
\newblock From stochastic, individual-based models to the canonical equation of
  adaptive dynamics in one step.
\newblock {\em Ann. Appl. Probab.}, 27(2):1093--1170, 2017.

\bibitem{BeBruShi16}
J.~Berestycki, E.~Brunet, and Z.~Shi.
\newblock The number of accessible paths in the hypercube.
\newblock {\em Bernoulli}, 22(2):653--680, 2016.

\bibitem{BeBruShi17}
J.~Berestycki, E.~Brunet, and Z.~Shi.
\newblock Accessibility percolation with backsteps.
\newblock {\em ALEA Lat. Am. J. Probab. Math. Stat.}, 14(1):45--62, 2017.

\bibitem{BoPa97}
B.~Bolker and S.~W. Pacala.
\newblock Using moment equations to understand stochastically driven spatial
  pattern formation in ecological systems.
\newblock {\em Theor. Popul. Biol.}, 52(3):179--197, 1997.

\bibitem{Bov21}
A.~Bovier.
\newblock Stochastic models for adaptive dynamics: Scaling limits and
  diversity.
\newblock In {\em Probabilistic Structures in Evolution, \emph{E. Baake and A.
  Wakolbinger, Eds.}}, volume~17 of {\em EMS Series of Congress Reports}, pages
  127--149. EMS Press, Berlin, 2021.

\bibitem{BoCoNeu18}
A.~Bovier, L.~Coquille, and R.~Neukirch.
\newblock The recovery of a recessive allele in a {M}endelian diploid model.
\newblock {\em J. Math. Biol.}, 77(4):971--1033, 2018.

\bibitem{BoCoSm19}
A.~Bovier, L.~Coquille, and C.~Smadi.
\newblock Crossing a fitness valley as a metastable transition in a stochastic
  population model.
\newblock {\em Ann. Appl. Probab.}, 29(6):3541--3589, 2019.

\bibitem{BovHol15}
A.~Bovier and F.~den Hollander.
\newblock {\em Metastability, A Potential-Theoretic Approach}, volume 351 of
  {\em Grundlehren der Mathematischen Wissenschaften}.
\newblock Springer Cham Heidelberg New York Dortrecht London, 2015.

\bibitem{Cha06}
N.~Champagnat.
\newblock A microscopic interpretation for adaptive dynamics trait substitution
  sequence models.
\newblock {\em Stoch. Process. Appl.}, 116(8):1127--1160, 2006.

\bibitem{ChMe11}
N.~Champagnat and S.~M\'{e}l\'{e}ard.
\newblock Polymorphic evolution sequence and evolutionary branching.
\newblock {\em Probab. Theory Relat. Fields}, 151(1-2):45--94, 2011.

\bibitem{ChMeTr19}
N.~Champagnat, S.~M{\'e}l{\'e}ard, and V.~C. Tran.
\newblock Stochastic analysis of emergence of evolutionary cyclic behavior in
  population dynamics with transfer.
\newblock {\em Ann. Appl. Probab.}, 31(4):1820--1867, 2021.

\bibitem{CiNa13}
E.~N.~M. Cirillo and F.~R. Nardi.
\newblock Relaxation height in energy landscapes: An application to multiple
  metastable states.
\newblock {\em J.Stat.Phys.}, 150:1080--1114, 2013.

\bibitem{CoKrSm21}
L.~Coquille, A.~Kraut, and C.~Smadi.
\newblock Stochastic individual-based models with power law mutation rate on a
  general finite trait space.
\newblock {\em Electron. J. Probab.}, 26:1--37, 2021.

\bibitem{DawGre14}
D.~A. Dawson and A.~Greven.
\newblock {\em Spatial Fleming-Viot models with selection and mutation}.
\newblock Lecture Notes in Mathematics, Vol.\ 2092. Springer International
  Publishing, 2014.

\bibitem{DeViKru14}
J.~A.~G. De~Visser and J.~Krug.
\newblock Empirical fitness landscapes and the predictability of evolution.
\newblock {\em Nat. Rev. Genet.}, 15(7):480--490, 2014.

\bibitem{EtKu86}
S.~N. Ethier and T.~G. Kurtz.
\newblock {\em Markov processes}.
\newblock Wiley Ser. in Probab. and Math. Stat. John Wiley \& Sons, Inc., New
  York, 1986.

\bibitem{FoMe04}
N.~Fournier and S.~M\'{e}l\'{e}ard.
\newblock A microscopic probabilistic description of a locally regulated
  population and macroscopic approximations.
\newblock {\em Ann. Appl. Probab.}, 14(4):1880--1919, 2004.

\bibitem{Gill76}
D.~T. Gillespie.
\newblock A general method for numerically simulating the stochastic time
  evolution of coupled chemical reactions.
\newblock {\em J. Comput. Phys.}, 22(4):403--434, 1976.

\bibitem{Gill84}
J.~H. Gillespie.
\newblock Molecular evolution over the mutational landscape.
\newblock {\em Evolution}, 38(5):1116--1129, 1984.

\bibitem{GoIwNoTr09}
C.~S. Gokhale, Y.~Iwasa, M.~A. Nowak, and A.~Traulsen.
\newblock The pace of evolution across fitness valleys.
\newblock {\em Journal of Theoretical Biology}, 259(3):613--620, 2009.

\bibitem{Jain07}
K.~Jain.
\newblock Evolutionary dynamics of the most populated genotype on rugged
  fitness landscapes.
\newblock {\em Phys. rev. E, Stat., nonlinear, and soft matter phys.}, 76 3 Pt
  1:031922, 2007.

\bibitem{JainKrug05}
K.~Jain and J.~Krug.
\newblock Evolutionary trajectories in rugged fitness landscapes.
\newblock {\em J Stat Mech: Theory and Exp.}, 2005(04):P04008, apr 2005.

\bibitem{JainKrug07}
K.~Jain and J.~Krug.
\newblock Deterministic and stochastic regimes of asexual evolution on rugged
  fitness landscapes.
\newblock {\em Genetics}, 175:1275--88, 03 2007.

\bibitem{Koma07}
N.~L. Komarova.
\newblock Loss- and gain-of-function mutations in cancer: Mass-action, spatial
  and hierarchical models.
\newblock {\em J. Stat. Phys.}, 128:413--446, 2007.

\bibitem{KrBo19}
A.~Kraut and A.~Bovier.
\newblock From adaptive dynamics to adaptive walks.
\newblock {\em J. Math. Biol.}, 79(5):1699--1747, 2019.

\bibitem{Krug21}
J.~Krug.
\newblock Accessibility percolation in random fitness landscapes.
\newblock In {\em Probabilistic Structures in Evolution, \emph{E. Baake and A.
  Wakolbinger, Eds.}}, volume~17 of {\em EMS Series of Congress Reports}, pages
  1--22. EMS Press, Berlin, 2021.

\bibitem{MarRai17}
I.~Martincorena, K.~M. Raine, M.~Gerstung, K.~J. Dawson, K.~Haase, P.~Van~Loo,
  H.~Davies, M.~R. Stratton, and P.~J. Campbell.
\newblock Universal patterns of selection in cancer and somatic tissues.
\newblock {\em Cell}, 171(5):1029--1041, 2017.

\bibitem{MeNiGe92}
J.~A.~J. Metz, R.~M. Nisbet, and S.~A.~H. Geritz.
\newblock How should we define 'fitness' for general ecological scenarios?
\newblock {\em Trends Ecol. Evol.}, 7(6):198--202, 1992.

\bibitem{KrugNeid2011}
J.~Neidhart and J.~Krug.
\newblock Adaptive walks and extreme value theory.
\newblock {\em Phys. Rev. Lett.}, 107:178102, Oct 2011.

\bibitem{NichAn19}
M.~Nicholson and T.~Antal.
\newblock Competing evolutionary paths in growing populations with applications
  to multidrug resistance.
\newblock {\em PLoS Comput Biol}, 15(4), 2019.

\bibitem{NoKr15}
S.~Nowak and J.~Krug.
\newblock Analysis of adaptive walks on {NK} fitness landscapes with different
  interaction schemes.
\newblock {\em J. Stat. Mech. Theory Exp.}, (6):P06014, 27, 2015.

\bibitem{Orr03}
H.~A. Orr.
\newblock A minimum on the mean number of steps taken in adaptive walks.
\newblock {\em J. Theor. Biol.}, 220(2):241--247, 2003.

\bibitem{PaMa06}
M.~N. Pallen, M.
\newblock From the origin of species to the origin of bacterial flagella.
\newblock {\em Nat. Rev. Microbiol.}, 4:784--790, 2006.

\bibitem{SchKr14}
B.~Schmiegelt and J.~Krug.
\newblock Evolutionary accessibility of modular fitness landscapes.
\newblock {\em J. Stat. Phys.}, 154(1-2):334--355, 2014.

\bibitem{Sma17}
C.~Smadi.
\newblock The effect of recurrent mutations on genetic diversity in a large
  population of varying size.
\newblock {\em Acta Appl. Math.}, 149:11--51, 2017.

\end{thebibliography}

\end{document}